\documentclass[reqno,11pt]{amsart}

\usepackage{a4wide}
\usepackage{color}
\usepackage{mathrsfs}
\usepackage{mathtools}
\usepackage{amsmath}
\usepackage{amssymb}
\usepackage{bbm}
\usepackage{esint}
\usepackage{nicefrac}
\numberwithin{equation}{section}
\usepackage[colorlinks,citecolor=green,linkcolor=red]{hyperref}

\usepackage[latin1]{inputenc}

\newcommand{\N}{\mathbb{N}}
\newcommand{\R}{\mathbb{R}}

\newcommand{\sfd}{{\sf d}}
\renewcommand{\d}{{\mathrm d}}
\newcommand{\e}{{\rm e}}
\newcommand{\X}{{\rm X}}
\newcommand{\Y}{{\rm Y}}

\newcommand{\mm}{\mathfrak{m}}
\newcommand{\1}{\mathbbm 1}

\newcommand{\LIP}{{\rm LIP}}

\newcommand{\lip}{{\rm lip}}

\newcommand{\ppi}{{\mbox{\boldmath\(\pi\)}}}
\newcommand{\eeta}{{\mbox{\boldmath\(\eta\)}}}
\newcommand{\sppi}{{\mbox{\scriptsize\boldmath\(\pi\)}}}

\newcommand{\limi}{\varliminf}
\newcommand{\lims}{\varlimsup}

\newcommand{\eps}{\varepsilon}

\newcommand{\fr}{\penalty-20\null\hfill\(\blacksquare\)}

\newtheorem{theorem}{Theorem}[section]
\newtheorem{corollary}[theorem]{Corollary}
\newtheorem{lemma}[theorem]{Lemma}
\newtheorem{proposition}[theorem]{Proposition}
\newtheorem{definition}[theorem]{Definition}

\newtheorem{remark}[theorem]{Remark}
\newtheorem{problem}[theorem]{Problem}

\title{Testing the Sobolev property with a single test plan}
\author{Enrico Pasqualetto}
\address{Department of Mathematics and Statistics,
P.O.\ Box 35 (MaD), FI-40014 University of Jyvaskyla}
\email{enrico.e.pasqualetto@jyu.fi}
\begin{document}
\date{\today}
\allowdisplaybreaks
\keywords{Sobolev space, test plan, \(\sf RCD\) space}
\subjclass[2010]{53C23, 46E35}
\begin{abstract}
We prove that in a vast class of metric measure spaces
(namely, those whose associated Sobolev space is separable) the following
property holds: a single test plan can be used to recover
the minimal weak upper gradient of any Sobolev function.
This means that, in order to identify which are the exceptional curves
in the weak upper gradient inequality, it suffices to consider the
negligible sets of a suitable Borel measure on curves,
rather than the ones of the \(p\)-modulus. Moreover, on \(\sf RCD\) spaces
we can improve our result, showing that the test plan can
be also chosen to be concentrated on an equi-Lipschitz family of curves.
\end{abstract}
\maketitle
\tableofcontents
\section*{Introduction}
\noindent
Throughout the past two decades, the classical theory of first-order
Sobolev spaces has been successfully generalised to the abstract setting
of metric measure spaces. Two strategies played a central role
in the development of this subject: the approximation by Lipschitz
functions (introduced by J.\ Cheeger \cite{Cheeger00}) and the analysis
of the behaviour along curves (proposed by N.\ Shanmugalingam
\cite{Shanmugalingam00}, and later revisited by L.\ Ambrosio, N.\ Gigli,
and G.\ Savar\'{e} \cite{AmbrosioGigliSavare11}). As it has been
eventually proven in \cite{AmbrosioGigliSavare11-3}, all these approaches
are fully equivalent.
\medskip

Let \((\X,\sfd)\) be a (complete, separable) metric space endowed
with a (boundedly finite) Borel measure \(\mm\). Let \(p\in(1,\infty)\)
be fixed. Then the \textbf{\(p\)-Sobolev space} \(W^{1,p}(\X)\)
is a Banach space whose elements \(f\) are associated with a minimal
object \(|Df|_p\in L^p(\mm)\), which is called the \textbf{minimal
\(p\)-relaxed slope} \cite{Cheeger00} or the \textbf{minimal \(p\)-weak
upper gradient} \cite{Shanmugalingam00,AmbrosioGigliSavare11}, and is the
smallest \(p\)-integrable function that bounds from above the (modulus of
the) variation of $f$. In Cheeger's approach, \(|Df|_p\) can be characterised
as the minimal possible strong \(L^p(\mm)\)-limit of \(\lip(f_n)\) among all
sequences \((f_n)_n\subseteq\LIP_{bs}(\X)\) with
\(\lim_n\|f-f_n\|_{L^p(\mm)}=0\), where \(\lip(f_n)\) stands for the slope
of \(f_n\) (see \eqref{eq:def_slope}). In duality with this `Eulerian'
relaxation procedure, it is possible -- from a more `Lagrangian' viewpoint --
to identify \(|Df|_p\) by looking at the behaviour of \(f\) along rectifiable
curves. Namely, \(|Df|_p\) is the minimal function \(G\in L^p(\mm)\)
such that for \emph{almost every} absolutely continuous curve
\(\gamma\) it holds that \(f\circ\gamma\) is absolutely continuous and
\begin{equation}\tag{\(\star\)}\label{eq:wug_intro}
\bigg|\frac{\d}{\d t}\,f(\gamma_t)\bigg|\leq G(\gamma_t)\,|\dot\gamma_t|
\quad\text{ for }\mathcal L^1\text{-a.e.\ }t\in[0,1].
\end{equation}
There are different ways to detect the negligible families of curves that
are excluded from the weak upper gradient condition \eqref{eq:wug_intro}.
In Shanmugalingam's approach, the exceptional curves are measured with respect
to the \textbf{\(p\)-modulus} \({\rm Mod}_p\), which is an outer measure on
paths that plays a crucial role in function theory; cf.\ \cite{HKST15}.
Ambrosio, Gigli, and Savar\'{e} proposed the alternative notion of test plan:
calling \(q\in(1,\infty)\) the conjugate exponent of \(p\), they define a
\textbf{\(q\)-test plan} on \((\X,\sfd,\mm)\) as a Borel probability measure
\(\ppi\) on \(AC\big([0,1],\X\big)\) such that
\begin{align*}
\exists C>0:\quad&(\e_t)_\#\ppi\leq C\mm
\quad\forall t\in[0,1],\\
&\int\!\!\!\int_0^1|\dot\gamma_t|^q\,\d t\,\d\ppi(\gamma)<+\infty,
\end{align*}
where the evaluation map \(\e_t\) is given by
\(\e_t(\gamma)\coloneqq\gamma_t\). The first condition is a compression
estimate -- which grants that the plan does not concentrate mass too
much at any time -- while the second one is an integral bound on the
speed of the curves selected by the plan. It is then possible to
express \(|Df|_p\) as the minimal \(G\in L^p(\mm)\) such that
for every \(q\)-test plan \(\ppi\) the inequality \eqref{eq:wug_intro}
holds for \(\ppi\)-a.e.\ \(\gamma\). There are two main differences
between the \(p\)-modulus and a \(q\)-test plan: firstly, the former
is an outer measure, while the latter is a \(\sigma\)-additive Borel
measure (but a priori one has to consider possibly uncountably many
test plans to identify the minimal weak upper gradient); secondly,
in the definition of test plan the parametrisation of the involved
curves plays an essential role, while the
modulus is parametrisation-invariant. The duality between modulus
and plans has been studied in \cite{ADMS13}.
\medskip

The aim of this paper is to show that we can find a single
\(q\)-test plan \(\ppi_q\) -- which we shall call the \textbf{master test plan}
-- that is sufficient to recover the minimal weak upper gradient
of any given Sobolev function. More precisely, for every \(f\in W^{1,p}(\X)\)
it holds that \(|Df|_p\) is the minimal \(G\in L^p(\mm)\) such that
\eqref{eq:wug_intro} holds for \(\ppi_q\)-a.e.\ \(\gamma\).
This result will be achieved on a vast class of metric measure spaces,
\emph{i.e.}, those having separable Sobolev space \(W^{1,p}(\X)\), which is
a quite mild assumption (cf.\ Remark \ref{rmk:W_sep}). Let us briefly
outline the ideas behind the proof:
\begin{itemize}
\item[\(\rm a)\)] The main tool we use is the \textbf{plan representing
the gradient} of a Sobolev function, a concept introduced by Gigli
in \cite{Gigli12}. This means, roughly speaking, that the `derivative'
at time \(t=0\) of the test plan coincides with the gradient of a given
function.
\item[\(\rm b)\)] In lack of a linear structure underlying the ambient
space \(\X\), we work within the framework of the abstract tensor
calculus built by Gigli in \cite{Gigli14}, which relies upon the theory
of \textbf{normed modules}. This supplies the functional-analytic tools
we will need.
\item[\(\rm c)\)] We will further investigate the plans representing a
gradient and fit them in the setting of the normed modules calculus,
which was still not available at the time of \cite{Gigli12}.
More precisely, we prove -- in a suitable sense -- that if a test plan \(\ppi\)
represents the gradient of \(f\in W^{1,p}(\X)\), then for every
\(g\in W^{1,p}(\X)\) and \(\ppi\)-a.e.\ \(\gamma\) the derivative
at time \(t=0\) of \(g\circ\gamma\) coincides with
\(\langle\nabla g,\nabla f\rangle(\gamma_0)\). See Proposition
\ref{prop:conv_sc_prod} for the precise statement.
\item[\(\rm d)\)] Given a dense sequence \((f_n)_n\) in \(W^{1,p}(\X)\)
and calling \(\ppi^n\) the plan representing the gradient of \(f_n\),
we show -- by using the results we mentioned in item c) -- that the
countable family \(\{\ppi^n\}_n\) of \(q\)-test plans is sufficient to
identify the minimal weak upper gradient of each Sobolev function.
Finally, by suitably combining the measures \(\ppi^n\) we obtain the
desired master test plan \(\ppi_q\). See Theorem \ref{thm:master_tp}
for the details.
\end{itemize}
The weak upper gradient condition \eqref{eq:wug_intro}
can be additionally used (when considered with respect to the modulus,
or to the totality of test plans) to detect which functions
are Sobolev. Currently, it is not known whether the same holds for the
master test plan; cf.\ Problem \ref{pb:open_pb}.
\medskip

In the last part of the paper, we improve our existence result of
master test plans in the case in which the metric measure space
\((\X,\sfd,\mm)\) satisfies a lower Ricci curvature bound. More
specifically, we consider the so-called \textbf{\(\sf RCD\) spaces}, which
are infinitesimally Hilbertian metric measure spaces (\emph{i.e.},
the associated \(2\)-Sobolev space is Hilbert) fulfilling the celebrated
curvature-dimension condition introduced by Lott--Sturm--Villani
\cite{Lott-Villani07,Sturm06I,Sturm06II}. In this framework, we show
that the master test plan \(\ppi_2\) can be also chosen to be concentrated
on an equi-Lipschitz family of curves; cf.\ Theorem \ref{thm:master_tp_RCD}.
This sort of property has to do with the dependence on the exponent \(p\)
of minimal \(p\)-weak upper gradients, see Remark \ref{rmk:infty_tp_diff_p}
for a more detailed discussion. To prove
Theorem \ref{thm:master_tp_RCD}, instead of plans representing
the gradient we employ the theory of \textbf{regular Lagrangian flows},
available on \(\sf RCD\) spaces thanks to \cite{Ambrosio-Trevisan14}.
\medskip

{\bf Acknowledgements.}
The author would like to thank Tapio Rajala and Daniele Semola
for the careful reading of a preliminary version of this manuscript.
This research has been supported by the Academy of Finland,
projects 274372, 307333, 312488, and 314789.
\section{Preliminaries}
\subsection{Sobolev calculus on metric measure spaces}\label{s:Sobolev}
For the purposes of this article, by \textbf{metric measure space}
we mean a triple \((\X,\sfd,\mm)\), where
\begin{align*}
(\X,\sfd)&\quad\text{ is a complete and separable metric space,}\\
\mm\geq 0&\quad\text{ is a boundedly finite Borel measure on }(\X,\sfd).
\end{align*}
The space \(C\big([0,1],\X\big)\) of continuous curves
in \(\X\) is a complete and separable metric space when
equipped with the supremum distance
\(\sfd_\infty(\gamma,\sigma)\coloneqq\max\big\{\sfd(\gamma_t,\sigma_t)
\,:\,t\in[0,1]\big\}\).
Given any \(t\in[0,1]\), we denote by \(\e_t\colon C\big([0,1],\X\big)\to\X\)
the \textbf{evaluation map} at time \(t\), namely, we set
\(\e_t(\gamma)\coloneqq\gamma_t\) for every
\(\gamma\in C\big([0,1],\X\big)\). Moreover, for any \(s,t\in[0,1]\)
with \(s<t\) we define the \textbf{restriction map}
\({\rm restr}_s^t\colon C\big([0,1],\X\big)\to C\big([0,1],\X\big)\)
as \({\rm restr}_s^t(\gamma)_r\coloneqq\gamma_{rt+(1-r)s}\).
Observe that both \(\e_t\) and \({\rm restr}_s^t\) are continuous maps.
A curve \(\gamma\in C\big([0,1],\X\big)\) is said to be
\textbf{absolutely continuous} provided there exists \(g\in L^1(0,1)\)
such that \(\sfd(\gamma_t,\gamma_s)\leq\int_s^t g(r)\,\d r\) for
every \(s,t\in[0,1]\) with \(s<t\). In this case, it holds that the limit
\(|\dot\gamma_t|\coloneqq\lim_{h\to 0}\sfd(\gamma_{t+h},\gamma_t)/|h|\)
exists at \(\mathcal L^1\)-a.e.\ \(t\in[0,1]\) and defines a
function in \(L^1(0,1)\), which is the minimal one (in the a.e.\ sense)
satisfying the inequality in the absolute continuity condition.
The function \(|\dot\gamma|\) -- which is declared to be \(0\) at those
\(t\in[0,1]\) where the above limit does not exist -- is called the
\textbf{metric speed} of \(\gamma\). We denote by \(AC\big([0,1],\X\big)\)
the family of all absolutely continuous curves on \(\X\). Given any
\(q\in(1,\infty)\), we define the family of \(q\)-absolutely continuous
curves as
\[
AC^q\big([0,1],\X\big)\coloneqq\Big\{\gamma\in AC\big([0,1],\X\big)\;
\Big|\;|\dot\gamma|\in L^q(0,1)\Big\}.
\]
The space of all real-valued Lipschitz functions on \((\X,\sfd)\) having
bounded support is denoted by \(\LIP_{bs}(\X)\). Given any function
\(f\in\LIP_{bs}(\X)\), we define its \textbf{slope}
\(\lip(f)\colon\X\to[0,+\infty)\) as
\begin{equation}\label{eq:def_slope}
\lip(f)(x)\coloneqq\lims_{y\to x}\frac{\big|f(x)-f(y)\big|}{\sfd(x,y)}
\quad\text{ if }x\in\X\text{ is an accumulation point}
\end{equation}
and \(\lip(f)(x)\coloneqq 0\) otherwise. Furthermore, for any
\(q\in(1,\infty)\) we denote by \(\mathscr P_q(\X)\) the set
of all Borel probability measures \(\mu\) on \((\X,\sfd)\) having
\textbf{finite \(q^{th}\)-moment}, \emph{i.e.}, satisfying
\[
\int\sfd^q(\cdot,\bar x)\,\d\mu<+\infty\quad
\text{ for some (thus any) point }\bar x\in\X.
\]
In the sequel, we will often consider the integral (in the sense
of Bochner \cite{DiestelUhl77}) of maps of the form
\([0,1]\ni t\mapsto\Phi_t\in\mathbb B\), where \(\mathbb B\) is
a separable Banach space; more precisely, \(\mathbb B\) will always
be an \(L^p\)-space, for some exponent \(p\in[1,\infty)\). The fact
that the maps \(\Phi\colon[0,1]\to\mathbb B\) we will consider are
strongly Borel follows by standard arguments, thus we will not
insist further on measurability issues. Let us just recall that
if a map \(\Phi\colon[0,1]\to L^p(\mu)\) is Bochner integrable,
then it holds that
\(\big(\int_0^1\Phi_t\,\d t\big)(x)=\int_0^1\Phi_t(x)\,\d t\)
for \(\mu\)-a.e.\ \(x\in\X\).
\begin{remark}\label{rmk:ex_tilde_m}{\rm
Let \((\X,\sfd,\mm)\) be a metric measure space. Fix any exponent
\(q\in(1,\infty)\). Then there exists a measure \(\tilde\mm\in\mathscr P_q(\X)\)
such that \(\mm\ll\tilde\mm\leq C\mm\) holds for some constant \(C>0\).

In order to prove it, fix any point \(\bar x\in\X\). Given that \((\X,\sfd)\)
is separable, we can find a sequence \((x_k)_k\subseteq\X\)
such that \(\X=\bigcup_{k\in\N}B_1(x_k)\). Recall that
\(\mm\big(B_1(x_k)\big)<+\infty\) for all \(k\in\N\).
We define \(A_1\coloneqq B_1(x_1)\) and
\(A_k\coloneqq B_1(x_k)\setminus(A_1\cup\ldots\cup A_{k-1})\)
for every \(k\geq 2\). Let us put
\[
\mu\coloneqq\sum_{k=1}^\infty\frac{\mm|_{A_k}}
{2^k\big(\sfd(x_k,\bar x)+1\big)^q\max\big\{\mm(A_k),1\big\}},
\qquad\tilde\mm\coloneqq\frac{\mu}{\mu(\X)}.
\]
It holds that \(\mu\) is a Borel measure on \(\X\) satisfying
\(\mu(\X)\leq\sum_{k=1}^\infty 2^{-k}=1\), whence \(\tilde\mm\)
is a (well-defined) Borel probability measure on \(\X\). If a Borel
set \(N\subseteq\X\) satisfies \(\mu(N)=0\), then we have that
\(\mm(N)=\sum_{k=1}^\infty\mm(N\cap A_k)=0\), thus showing that
\(\mm\ll\tilde\mm\). Moreover, observe that one has \(\mu\leq\sum_{k=1}^\infty
2^{-k}\,\mm|_{A_k}\leq\mm\) and accordingly \(\tilde\mm\leq\mu(\X)^{-1}\mm\).
Finally, given that the inequality
\(\sfd(\cdot,\bar x)\leq\sfd(x_k,\bar x)+1\)
holds on \(A_k\) for any \(k\in\N\), we conclude that
\[
\int\sfd^q(\cdot,\bar x)\,\d\tilde\mm=\frac{1}{\mu(\X)}\sum_{k=1}^\infty
\frac{1}{2^k\max\big\{\mm(A_k),1\big\}}
\int_{A_k}\bigg(\frac{\sfd(\cdot,\bar x)}{\sfd(x_k,\bar x)+1}\bigg)^q
\d\mm\leq\frac{1}{\mu(\X)},
\]
thus proving that the measure \(\tilde\mm\) has finite \(q^{th}\)-moment.
\fr}\end{remark}
\subsubsection{Definition of Sobolev space}
Let us recall Cheeger's notion of Sobolev space, based upon
the relaxation of the slope. Other approaches will be discussed
in Sections \ref{ss:Mod} and \ref{ss:test_plans}.
\begin{definition}[Sobolev space \cite{Cheeger00}]
Let \((\X,\sfd,\mm)\) be a metric measure space and \(p\in(1,\infty)\).
Then we declare that a function \(f\in L^p(\mm)\) belongs to the
\textbf{\(p\)-Sobolev space} \(W^{1,p}(\X)\) provided there exists a sequence
\((f_n)_n\subseteq\LIP_{bs}(\X)\) such that \(f_n\to f\) in \(L^p(\mm)\) and
\[
\limi_{n\to\infty}\int\lip^p(f_n)\,\d\mm<+\infty.
\]
\end{definition}
The Sobolev space \(W^{1,p}(\X)\) is a Banach space if endowed with the norm
\[
\|f\|_{W^{1,p}(\X)}\coloneqq\Big(\|f\|_{L^p(\mm)}^p+
p\,{\rm E}_{{\rm Ch},p}(f)\Big)^{\frac{1}{p}}
\quad\text{ for every }f\in W^{1,p}(\X),
\]
where the \textbf{Cheeger \(p\)-energy} \({\rm E}_{{\rm Ch},p}\) is given by
\[
{\rm E}_{{\rm Ch},p}(f)\coloneqq\inf\bigg\{\limi_{n\to\infty}
\frac{1}{p}\int\lip^p(f_n)\,\d\mm\;\bigg|\;(f_n)_n\subseteq\LIP_{bs}(\X),
\,f_n\to f\text{ in }L^p(\mm)\bigg\}.
\]
It holds that for every \(f\in W^{1,p}(\X)\) there exists a unique function
\(|Df|_p\in L^p(\mm)\) such that
\[
{\rm E}_{{\rm Ch},p}(f)=\frac{1}{p}\int|Df|_p^p\,\d\mm.
\]
The function \(|Df|_p\) is called the \textbf{minimal \(p\)-relaxed slope}
of \(f\).
\begin{remark}\label{rmk:ineq_diff_p}{\rm
The minimal \(p\)-relaxed slope might depend on the exponent \(p\). More
precisely, if \(p,p'\in(1,\infty)\) and \(f\in W^{1,p}(\X)\cap W^{1,p'}(\X)\),
then it might happen that \(|Df|_p\neq|Df|_{p'}\). Some examples of spaces in
which this phenomenon occurs can be found in \cite{DiMarinoSpeight15}.
\fr}\end{remark}
\begin{remark}\label{rmk:W_sep}{\rm
The reflexivity properties of the Sobolev spaces are investigated
in \cite{ACDM15}, where the authors proved, \emph{e.g.}, that \(W^{1,p}(\X)\)
is reflexive as soon as the underlying space \((\X,\sfd,\mm)\) is metrically
doubling. However, just one example of non-reflexive Sobolev space is known
(also provided in \cite{ACDM15}). Furthermore, the reflexivity of
\(W^{1,p}(\X)\) implies its separability.
\fr}\end{remark}
\subsubsection{The theory of normed modules}
We need to recall a few basic notions in the theory of normed
modules introduced in \cite{Gigli14,Gigli17}. Given a metric
measure space \((\X,\sfd,\mm)\) and an exponent \(p\in(1,\infty)\),
we say that \(\mathscr M\) is a
\textbf{\(L^p(\mm)\)-normed \(L^\infty(\mm)\)-module}
if it is a module over the ring \(L^\infty(\mm)\) and it is equipped
with a \textbf{pointwise norm} \(|\cdot|\colon\mathscr M\to L^p(\mm)\)
satisfying
\begin{align*}
|v|\geq 0&\quad\text{ for every }
v\in\mathscr M,\text{ with }|v|=0\text{ if and only if }v=0,\\
|f\cdot v|=|f||v|&\quad\text{ for every }
v\in\mathscr M\text{ and }f\in L^\infty(\mm),\\
|v+w|\leq|v|+|w|&\quad\text{ for every }v,w\in\mathscr M,
\end{align*}
where equalities and inequalities are intended in the \(\mm\)-a.e.\ sense.
Moreover, we require the norm
\(\|v\|_{\mathscr M}\coloneqq\big\||v|\big\|_{L^p(\mm)}\) to be complete,
whence \(\mathscr M\) has a Banach space structure.
\medskip

The \textbf{dual} of \(\mathscr M\) is given by the
space \(\mathscr M^*\) of \(L^\infty(\mm)\)-linear
continuous maps \(T\colon\mathscr M\to L^1(\mm)\).
Choosing \(q\in(1,\infty)\) so that \(\frac{1}{p}+\frac{1}{q}=1\),
we have that \(\mathscr M^*\) is a \(L^q(\mm)\)-normed
\(L^\infty(\mm)\)-module if endowed with the following
pointwise norm operator:
\[
|T|\coloneqq{\rm ess\,sup}\Big\{\big|T(v)\big|\;\Big|\;
v\in\mathscr M,\,|v|\leq 1\;\mm\text{-a.e.}\Big\}
\in L^q(\mm)\quad\text{ for every }T\in\mathscr M^*.
\]
The link between the Sobolev calculus and the theory of normed modules
is represented by the \textbf{cotangent module} \(L^p(T^*\X)\). It is
a \(L^p(\mm)\)-normed \(L^\infty(\mm)\)-module that comes with a
linear differential operator \(\d_p\colon W^{1,p}(\X)\to L^p(T^*\X)\)
and is characterised by these two properties:
\begin{align*}
|\d_p f|=|Df|_p\;\;\;\mm\text{-a.e.}&\quad\text{ for every }f\in W^{1,p}(\X),\\
\bigg\{\sum_{i=1}^n g_i\cdot\d_p f_i\;\bigg|\;(g_i)_{i=1}^n\subseteq
L^\infty(\mm),\,(f_i)_{i=1}^n\subseteq W^{1,p}(\X)\bigg\}&\quad
\text{ is dense in }L^p(T^*\X).
\end{align*}
The existence of the cotangent module when \(p=2\)
is proven in \cite{Gigli14}, while the case \(p\neq 2\) is treated
in \cite{GIGLI2019}. The dual \(L^q(T\X)\) of the space \(L^p(T^*\X)\)
is called the \textbf{tangent module}.
\medskip

Given any \(L^p(\mm)\)-normed \(L^\infty(\mm)\)-module \(\mathscr M\),
we define the map \({\sf Dual}\colon\mathscr M\to 2^{\mathscr M^*}\) as
\begin{equation}\label{eq:def_Dual}
{\sf Dual}(v)\coloneqq\Big\{\omega\in\mathscr M^*\;\Big|
\;\omega(v)=|v|^p=|\omega|^q\;\mm\text{-a.e.}\Big\}\quad
\text{ for every }v\in\mathscr M.
\end{equation}
It holds that \({\sf Dual}(v)\neq\emptyset\) for every \(v\in\mathscr M\),
as a consequence of Hahn--Banach theorem.
\medskip

Another important construction is that of \textbf{pullback module}.
Let \((\X,\sfd_\X,\mm_\X)\), \((\Y,\sfd_\Y,\mm_\Y)\) be metric measure
spaces. Let \(\varphi\colon\X\to\Y\) be a Borel map satisfying
\(\varphi_\#\mm_\X\leq C\mm_\Y\) for some constant \(C>0\).
Then it holds that for any \(L^p(\mm_\Y)\)-normed \(L^\infty(\mm_\Y)\)-module
\(\mathscr M\) there exist a unique \(L^p(\mm_\X)\)-normed
\(L^\infty(\mm_\X)\)-module \(\varphi^*\mathscr M\) and a unique
linear map \(\varphi^*\colon\mathscr M\to\varphi^*\mathscr M\) such that
the following properties are satisfied:
\begin{align*}
|\varphi^*v|=|v|\circ\varphi\;\;\;\mm_\X\text{-a.e.}&
\quad\text{ for every }v\in\mathscr M,\\
\bigg\{\sum_{i=1}^n f_i\cdot\varphi^*v_i\;\bigg|\;(f_i)_{i=1}^n\subseteq
L^\infty(\mm_\X),\,(v_i)_{i=1}^n\subseteq\mathscr M\bigg\}&
\quad\text{ is dense in }\varphi^*\mathscr M.
\end{align*}
We refer the reader to \cite{Gigli14,Gigli17} for a complete account
about normed modules.
\subsubsection{Infinitesimal Hilbertianity}
Let \((\X,\sfd,\mm)\) be a metric measure space. A \(L^2(\mm)\)-normed
\(L^\infty(\mm)\)-module \(\mathscr M\) is said to be a \textbf{Hilbert module}
provided the parallelogram rule holds:
\begin{equation}\label{eq:parall_rule}
|v+w|^2+|v-w|^2=2\,|v|^2+2\,|w|^2\;\;\;\mm\text{-a.e.}
\quad\text{ for every }v,w\in\mathscr M.
\end{equation}
The condition in \eqref{eq:parall_rule} is equivalent to requiring
that \(\mathscr M\) is Hilbert when viewed as a Banach space.
The \textbf{pointwise scalar product}
\(\langle\cdot,\cdot\rangle\colon\mathscr M\times\mathscr M\to L^1(\mm)\)
is then defined as follows:
\[
\langle v,w\rangle\coloneqq\frac{|v+w|^2-|v|^2-|w|^2}{2}\;\;\;\mm\text{-a.e.}
\quad\text{ for every }v,w\in\mathscr M.
\]
It can be straightforwardly checked that the map
\(\langle\cdot,\cdot\rangle\) is \(L^\infty(\mm)\)-bilinear and continuous.
\begin{remark}{\rm
Consider a \(L^2(\mm)\)-normed \(L^\infty(\mm)\)-module \(\mathscr M\)
and the map \({\sf Dual}\colon\mathscr M\to 2^{\mathscr M^*}\) as in
\eqref{eq:def_Dual}. Then \(\mathscr M\) is a Hilbert module if and
only if \(\sf Dual\) is single-valued and the unique element of
\({\sf Dual}(v)\) linearly depends on \(v\in\mathscr M\). The map
associating to every \(v\in\mathscr M\) the unique element
\({\sf R}_{\mathscr M}(v)\in\mathscr M^*\) of \({\sf Dual}(v)\)
is called the \textbf{Riesz isomorphism} of \(\mathscr M\). Moreover,
it holds that \({\sf R}_{\mathscr M}\colon\mathscr M\to\mathscr M^*\)
is a linear isomorphism that preserves the pointwise norm.
The above claims can be proven by arguing as in \cite[Exercise 4.2.11]{GP20}.
\fr}\end{remark}
A metric measure space \((\X,\sfd,\mm)\) is said to be
\textbf{infinitesimally Hilbertian} \cite{Gigli12} provided
the \(2\)-Sobolev space \(W^{1,2}(\X)\) is a Hilbert space, or
equivalently the cotangent module \(L^2(T^*\X)\) is a Hilbert module.
In this case, we define the linear operator
\(\nabla\colon W^{1,2}(\X)\to L^2(T\X)\) as
\[
\nabla f\coloneqq{\sf R}_{L^2(T^*\X)}(\d_2 f)\in L^2(T\X)
\quad\text{ for every }f\in W^{1,2}(\X).
\]
We say that \(\nabla f\) is the \textbf{gradient} of the function \(f\).
\subsection{Modulus and Newtonian space}\label{ss:Mod}
The notion of Sobolev space that we described in Section \ref{s:Sobolev}
corresponds, in the smooth framework, to the approach via approximation
by smooth functions. Another viewpoint on weakly differentiable functions
in the Euclidean space is the one introduced by B.\ Levi \cite{Levi06},
which consists in checking the behaviour of functions along curves.
This approach has been further refined by B.\ Fuglede \cite{Fuglede57},
who made it frame-independent by using the potential-theoretic notion of
\textbf{modulus}. Later on, the theory has been extended by
N.\ Shanmugalingam \cite{Shanmugalingam00} to the setting of metric
measure spaces, by introducing the so-called \textbf{Newtonian space},
whose definition builds upon the notion of \textbf{upper gradient} introduced
by J.\ Heinonen and P.\ Koskela \cite{HK98}.
\medskip

Let \((\X,\sfd,\mm)\) be a given metric measure space. Given an exponent
\(p\in(1,\infty)\) and any family \(\Gamma\subseteq AC\big([0,1],\X\big)\)
of non-constant curves, we define the \textbf{\(p\)-modulus} of \(\Gamma\) as
\[
{\rm Mod}_p(\Gamma)\coloneqq\inf\bigg\{\int\rho^p\,\d\mm\;\bigg|\;
\rho\colon\X\to[0,+\infty]\text{ Borel,}\,\int_0^1\rho(\gamma_t)\,
|\dot\gamma_t|\,\d t\geq 1\text{ for every }\gamma\in\Gamma\bigg\}.
\]
It holds that \({\rm Mod}_p\) is an outer measure. Typically, it is defined
on all (non-parametric) curves, but here we prefer this formulation since
it better fits our approach. A property is said to hold
\textbf{\({\rm Mod}_p\)-almost everywhere} provided it is satisfied
by every \(\gamma\) in some set \(\Gamma\) of curves whose complement is
\({\rm Mod}_p\)-negligible. Given two Borel functions \(\bar f\colon\X\to\R\)
and \(G\colon\X\to[0,+\infty]\) with \(G\in L^p(\mm)\), we say that \(G\)
is a \textbf{\(p\)-weak upper gradient} of \(\bar f\) if for
\({\rm Mod}_p\)-a.e.\ curve \(\gamma\) it holds that
\(\bar f\circ\gamma\) is absolutely continuous and
\(\big|\frac{\d}{\d t}\bar f(\gamma_t)\big|\leq G(\gamma_t)\,|\dot\gamma_t|\)
for \(\mathcal L^1\)-a.e.\ \(t\in[0,1]\).
\begin{definition}[Newtonian space \cite{Shanmugalingam00}]
Let \((\X,\sfd,\mm)\) be a metric measure space. Fix any
exponent \(p\in(1,\infty)\). Then the \textbf{Newtonian space}
\(N^{1,p}(\X)\) is the family of all \(f\in L^p(\mm)\) that admit
a Borel representative \(\bar f\colon\X\to\R\) having a \(p\)-weak
upper gradient \(G\in L^p(\mm)\).
\end{definition}

The Newtonian space can be made into a Banach space: given any
\(f\in N^{1,p}(\X)\), we define
\[
\|f\|_{N^{1,p}(\X)}\coloneqq\Big(\|f\|_{L^p(\mm)}^p+
\inf_{G\in D_p[f]}\|G\|_{L^p(\mm)}^p\Big)^{\frac{1}{p}},
\]
where \(D_p[f]\) stands for the family of all Borel functions
\(G\colon\X\to[0,+\infty]\) that are \(p\)-weak upper gradients
of some Borel version of \(f\). It turns out that \(\|\cdot\|_{N^{1,p}(\X)}\)
is a complete norm on \(N^{1,p}(\X)\). There exists a unique function
\(G_{f,p}\in D_p[f]\) having minimal \(L^p(\mm)\)-norm among all elements of
\(D_p[f]\), and it is minimal also in the \(\mm\)-a.e.\ sense. It holds that:
\begin{proposition}
Let \((\X,\sfd,\mm)\) be a metric measure space. Let \(p\in(1,\infty)\)
be given. Then we have that \(W^{1,p}(\X)\subseteq N^{1,p}(\X)\) and
\(G_{f,p}\leq|Df|_p\) holds \(\mm\)-a.e.\ for every \(f\in W^{1,p}(\X)\).
\end{proposition}
We refer the reader to the monograph \cite{HKST15} for a thorough
discussion about this topic.
\subsection{Test plans}\label{ss:test_plans}
To prove the equivalence between \(W^{1,p}(\X)\)
and \(N^{1,p}(\X)\), L.\ Ambrosio, N.\ Gigli, and G.\ Savar\'{e} introduced
in \cite{AmbrosioGigliSavare11,AmbrosioGigliSavare11-3} the notion
of \textbf{test plan}, which furnishes a more `probabilistic' way to
measure the exceptional curves in the weak upper gradient condition.
\medskip

Let \((\X,\sfd,\mm)\) be a metric measure space. Given any \(q\in(1,\infty)\)
and \(t\in(0,1]\), following \cite{Gigli12} we define the
\textbf{\(q\)-energy functional}
\({\rm E}_{q,t}\colon C\big([0,1],\X\big)\to[0,+\infty]\) as
\[
{\rm E}_{q,t}(\gamma)\coloneqq t\bigg(\fint_0^t|\dot\gamma_s|^q\,\d s\bigg)
^{\frac{1}{q}}\quad\text{ if }\gamma\in AC^q\big([0,1],\X\big)
\]
and \({\rm E}_{q,t}(\gamma)\coloneqq +\infty\) otherwise.
It can be readily checked that \({\rm E}_{q,t}\) is a Borel mapping.
\begin{definition}[Test plan \cite{AmbrosioGigliSavare11-3}]
\label{def:test_plan}
Let \((\X,\sfd,\mm)\) be a metric measure space and \(q\in(1,\infty)\).
Then a Borel probability measure \(\ppi\) on \(C\big([0,1],\X\big)\) is
said to be a \textbf{\(q\)-test plan} on \((\X,\sfd,\mm)\) provided:
\begin{itemize}
\item[\(\rm i)\)] There exists a constant \(C>0\) such that
\((\e_t)_\#\ppi\leq C\mm\) holds for every \(t\in[0,1]\). The
minimal such \(C\) is denoted by \({\rm Comp}(\ppi)>0\)
and called the \textbf{compression constant}.
\item[\(\rm ii)\)] The measure \(\ppi\) has
\textbf{finite kinetic \(q\)-energy}, which means that
\[
{\rm KE}_q(\ppi)\coloneqq\int{\rm E}_{q,1}(\gamma)^q\,\d\ppi(\gamma)
<+\infty.
\]
In particular, it holds that \(\ppi\) is concentrated
on \(AC^q\big([0,1],\X\big)\).
\end{itemize}
Also, we say that a Borel probability measure \(\ppi\) on
\(C\big([0,1],\X\big)\) is a \textbf{\(\infty\)-test plan} on
\((\X,\sfd,\mm)\) provided it satisfies item \emph{i)} and
it is concentrated on an equi-Lipschitz family of curves.
\end{definition}
Observe that if \(q,q'\in(1,\infty]\) satisfy \(q'\leq q\), then
every \(q\)-test plan is a \(q'\)-test plan. Moreover, given a
\(q\)-test plan \(\ppi\) and \(s,t\in[0,1]\) with \(s<t\),
it holds that \(({\rm rest}_s^t)_\#\ppi\) is a \(q\)-test plan.
\medskip

The relation between test plans and modulus has been deeply
investigated in \cite{ADMS13}. The following result
(whose proof can be found, \emph{e.g.}, in \cite[Lemma 2.2.26]{GP20})
is sufficient for the purposes of this paper. Being the formulation
slightly different, we report here also its proof.
\begin{lemma}\label{lem:Mod_vs_tp}
Let \((\X,\sfd,\mm)\) be a metric measure space. Let \(p,q\in(1,\infty)\)
be such that \(\frac{1}{p}+\frac{1}{q}=1\). Fix a \(q\)-test plan
\(\ppi\) and a family \(\Gamma\subseteq AC\big([0,1],\X\big)\)
of non-constant curves with \({\rm Mod}_p(\Gamma)=0\). Then there
exists a Borel set \(N\subseteq AC\big([0,1],\X\big)\)
such that \(\Gamma\subseteq N\) and \(\ppi(N)=0\).
\end{lemma}
\begin{proof}
For any \(n\in\N\), there is a Borel function
\(\rho_n\colon\X\to[0,+\infty]\) such that
\(\int_0^1\rho_n(\gamma_t)\,|\dot\gamma_t|\,\d t\geq 1\)
for every \(\gamma\in\Gamma\) and \(\int\rho_n^p\,\d\mm\leq 1/n\).
Since \((\gamma,t)\mapsto\rho_n(\gamma_t)\,|\dot\gamma_t|\) is a
Borel function, we have that the set \(N_n\coloneqq\big\{\gamma\,:\,
\int_0^1\rho_n(\gamma_t)\,|\dot\gamma_t|\,\d t\geq 1\big\}\) is Borel.
Therefore, the Borel set \(N\coloneqq\bigcap_n N_n\) contains
\(\Gamma\) and satisfies
\begin{align*}
\ppi(N)&\leq\inf_{n\in\N}\ppi(N_n)=
\inf_{n\in\N}\int\1_{N_n}(\gamma)\,\d\ppi(\gamma)
\leq\inf_{n\in\N}\int\!\!\!\int_0^1\rho_n(\gamma_t)\,|\dot\gamma_t|
\,\d t\,\d\ppi(\gamma)\\
&\leq\inf_{n\in\N}\bigg(\int\!\!\!\int_0^1\rho_n^p\circ\e_t\,\d t
\,\d\ppi\bigg)^{\frac{1}{p}}\bigg(\int\!\!\!\int_0^1|\dot\gamma_t|^q\,\d t
\,\d\ppi(\gamma)\bigg)^{\frac{1}{q}}\\
&\leq{\rm Comp}(\ppi)^{\frac{1}{p}}\,{\rm KE}_q(\ppi)^{\frac{1}{q}}
\inf_{n\in\N}\frac{1}{n^{1/p}}=0,
\end{align*}
thus proving the statement.
\end{proof}

A proof of the following continuity result
can be found, \emph{e.g.}, in \cite[Proposition 2.1.4]{GP20}.
\begin{proposition}\label{prop:cont_f_et}
Let \((\X,\sfd,\mm)\) be a metric measure space.
Let \(q\in(1,\infty)\) and \(r\in[1,\infty)\) be given.
Let \(\ppi\) be a \(q\)-test plan on \((\X,\sfd,\mm)\).
Then for any function \(f\in L^r(\mm)\) it holds that
\[
[0,1]\ni t\longmapsto f\circ\e_t\in L^r(\mm)
\quad\text{ is a strongly continuous map.}
\]
\end{proposition}
\subsubsection{\texorpdfstring{\((\Pi,p)\)}{Pi}-weak upper gradients}
We now focus on the role that test plans play in the Sobolev theory.
The key point is that they can be used to select the `negligible
families of curves' in the weak upper gradient condition, as we
are going to explain in the next definition.
\begin{definition}[\((\Pi,p)\)-weak upper gradient]
Let \((\X,\sfd,\mm)\) be a metric measure space and let \(p,q\in(1,\infty)\)
satisfy \(\frac{1}{p}+\frac{1}{q}=1\). Let \(\Pi\) be a family of
\(q\)-test plans on \(\X\). Let \(f\in L^p(\mm)\) be given. Then we declare
that a function \(G\in L^p(\mm)\) is a \textbf{\((\Pi,p)\)-weak upper gradient}
of \(f\) provided for any \(\ppi\in\Pi\) it holds that
\(f\circ\gamma\in W^{1,1}(0,1)\) for
\(\ppi\)-a.e.\ \(\gamma\in AC^q\big([0,1],\X\big)\) and
\[
\bigg|\frac{\d}{\d t}\,f(\gamma_t)\bigg|\leq G(\gamma_t)\,|\dot\gamma_t|
\quad\text{ for }(\ppi\otimes\mathcal L^1)\text{-a.e.\ }
(\gamma,t)\in AC^q\big([0,1],\X\big)\times[0,1].
\]
We denote by \({\rm G}_{\Pi,p}(f)\) the collection of all \((\Pi,p)\)-weak upper
gradients of \(f\). Also, we define
\[
W^{1,p}_\Pi(\X)\coloneqq\big\{f\in L^p(\mm)\;\big|\;
{\rm G}_{\Pi,p}(f)\neq\emptyset\big\}.
\]
\end{definition}

Observe that \(N^{1,p}(\X)\subseteq W^{1,p}_\Pi(\X)\)
and \(D_p[f]\subseteq{\rm G}_{\Pi,p}(f)\) for all \(f\in N^{1,p}(\X)\)
by Lemma \ref{lem:Mod_vs_tp}.
\begin{lemma}
Let \((\X,\sfd,\mm)\) be a metric measure space and \(p,q\in(1,\infty)\)
satisfy \(\frac{1}{p}+\frac{1}{q}=1\). Let \(\Pi\) be a family
of \(q\)-test plans on \(\X\). Then it holds that the set
\({\rm G}_{\Pi,p}(f)\) is a closed convex lattice of \(L^p(\mm)\)
for every \(f\in W^{1,p}_\Pi(\X)\).
\end{lemma}
\begin{proof}
Let \(f\in W^{1,p}_\Pi(\X)\) be fixed. Clearly, if
\(G_1,G_2\in{\rm G}_{\Pi,p}(f)\), then \(\min\{G_1,G_2\}\in{\rm G}_{\Pi,p}(f)\)
as well. Now fix a sequence \((G_n)_n\subseteq{\rm G}_{\Pi,p}(f)\) such
that \(G_n\to G\in L^p(\mm)\) strongly in \(L^p(\mm)\).
Up to a not relabelled subsequence, we have that \(G_n\to G\)
holds in the pointwise \(\mm\)-a.e.\ sense. Given any \(t\in[0,1]\), it follows
from the assumption \((\e_t)_\#\ppi\ll\mm\) that \(G_n\circ\e_t\to G\circ\e_t\)
holds pointwise \(\ppi\)-a.e.\ as \(n\to\infty\). Also, by Fubini
theorem we see that for \(\mathcal L^1\)-a.e.\ \(t\in[0,1]\) we have
\begin{equation}\label{eq:G_lattice_aux}
\bigg|\frac{\d}{\d t}\,f(\gamma_t)\bigg|\leq G_n(\gamma_t)\,|\dot\gamma_t|
\quad\text{ for every }n\in\N\text{ and }\ppi\text{-a.e.\ }
\gamma\in AC^q\big([0,1],\X\big).
\end{equation}
By letting \(n\to\infty\) in \eqref{eq:G_lattice_aux}, we get that
for \(\mathcal L^1\)-a.e.\ \(t\in[0,1]\) the inequality
\(\big|\frac{\d}{\d t}f(\gamma_t)\big|\leq G(\gamma_t)\,|\dot\gamma_t|\)
is satisfied for \(\ppi\)-a.e.\ \(\gamma\). By using Fubini theorem again, we
conclude that \(G\in{\rm G}_{\Pi,p}(f)\). This shows that \({\rm G}_{\Pi,p}(f)\)
is strongly closed in \(L^p(\mm)\), thus completing the proof of the statement.
\end{proof}
\begin{definition}[Minimal \((\Pi,p)\)-weak upper gradient]
Let \((\X,\sfd,\mm)\) be a metric measure space and \(p,q\in(1,\infty)\)
satisfy \(\frac{1}{p}+\frac{1}{q}=1\). Let \(\Pi\) be a family of \(q\)-test
plans on \(\X\). Fix any function \(f\in W^{1,p}_\Pi(\X)\). Then the (unique)
minimal element of \({\rm G}_{\Pi,p}(f)\) is denoted by \(|Df|_{\Pi,p}\) and
called the \textbf{minimal \((\Pi,p)\)-weak upper gradient} of \(f\).
\end{definition}
Whenever \(\Pi=\{\ppi\}\) is a singleton, we use the shorthand notation
\(W^{1,p}_\sppi(\X)\) and \(|Df|_{\sppi,p}\).
\begin{remark}{\rm
Observe that \(|Df|_{\Pi,p}\leq G_{f,p}\) holds \(\mm\)-a.e.\ for
every \(f\in N^{1,p}(\X)\).
\fr}\end{remark}

As already mentioned above, by considering the totality of test plans it
is possible to recover both the Sobolev space and the minimal relaxed
slope of each Sobolev function:
\begin{theorem}[Sobolev space via test plans \cite{AmbrosioGigliSavare11-3}]
\label{thm:density_lip}
Let \((\X,\sfd,\mm)\) be a metric measure space.
Let \(p,q\in(1,\infty)\) be such that \(\frac{1}{p}+\frac{1}{q}=1\).
Denote by \(\Pi_q\) the family of all \(q\)-test plans on \(\X\).
Then it holds that \(W^{1,p}_{\Pi_q}(\X)=W^{1,p}(\X)\) and
\[
|Df|_{\Pi_q,p}=|Df|_p\quad\text{ for every }f\in W^{1,p}(\X).
\]
In particular, it holds that \(N^{1,p}(\X)=W^{1,p}(\X)\)
and \(G_{f,p}=|Df|_p\) for every \(f\in W^{1,p}(\X)\).
\end{theorem}
\begin{proposition}\label{prop:ineq_mwug}
Let \((\X,\sfd,\mm)\) be a metric measure space and \(p,q\in(1,\infty)\)
satisfy \(\frac{1}{p}+\frac{1}{q}=1\). Let \(\Pi\subseteq\Pi'\)
be two given families of \(q\)-test plans on \(\X\). Then it holds that
\(W^{1,p}_{\Pi'}(\X)\subseteq W^{1,p}_\Pi(\X)\) and the inequality
\(|Df|_{\Pi,p}\leq|Df|_{\Pi',p}\) is satisfied \(\mm\)-a.e.\ for every
\(f\in W^{1,p}_{\Pi'}(\X)\). In particular, it holds
that \(W^{1,p}(\X)\subseteq W^{1,p}_\Pi(\X)\) and
\[
|Df|_{\Pi,p}\leq|Df|_p\;\;\;\mm\text{-a.e.}
\quad\text{ for every }f\in W^{1,p}(\X).
\]
\end{proposition}
\begin{proof}
To prove the first part of the claim, it suffices to observe
that any \((\Pi',p)\)-weak upper gradient is a \((\Pi,p)\)-weak upper gradient,
thus \(W^{1,p}_{\Pi'}(\X)\subseteq W^{1,p}_\Pi(\X)\) and for any
\(f\in W^{1,p}_{\Pi'}(\X)\) the function \(|Df|_{\Pi',p}\) is a
\((\Pi,p)\)-weak upper gradient of \(f\). Consequently, the last
part of the statement follows from the first one by recalling Theorem
\ref{thm:density_lip}.
\end{proof}
\subsubsection{Plans representing a gradient}
A special class of test plans is that of
\textbf{plans representing a gradient}, which have been introduced by
N.\ Gigli in \cite{Gigli12}. Roughly speaking, they are test plans
whose derivative at time \(0\) coincides with the gradient of a given
Sobolev function, in some generalised sense. These objects will play
a fundamental role in this paper.
\begin{definition}[Test plan representing a gradient \cite{Gigli12}]
\label{def:tp_repr_grad}
Let \((\X,\sfd,\mm)\) be a metric measure space. Let \(p,q\in(1,\infty)\)
satisfy \(\frac{1}{p}+\frac{1}{q}=1\). Let \(f\in W^{1,p}(\X)\) be given.
Then a \(q\)-test plan \(\ppi\) is said to \textbf{\(q\)-represent the gradient
of \(f\)} provided the following properties hold:
\begin{subequations}
\begin{align}\label{eq:tp_repr_grad_claim1}
\frac{f\circ\e_t-f\circ\e_0}{{\rm E}_{q,t}}&\to|Df|_p\circ\e_0
\quad\text{ strongly in }L^p(\ppi)\text{ as }t\searrow 0,\\
\label{eq:tp_repr_grad_claim2}
\bigg(\frac{{\rm E}_{q,t}}{t}\bigg)^{\frac{q}{p}}&\to
|Df|_p\circ\e_0\quad\text{ strongly in }L^p(\ppi)\text{ as }t\searrow 0.
\end{align}
\end{subequations}
\end{definition}
\begin{remark}{\rm
The above definition of test plan representing a gradient is slightly
different from the one introduced in \cite{Gigli12}.
First of all, a plan \(\ppi\) representing a gradient in the sense
of \cite{Gigli12} is not necessarily a test plan; however,
for some \(t\in(0,1)\) it holds that \(({\rm restr}_0^t)_\#\ppi\)
is a test plan on \(\X\). Also, the approach we chose is not the original one
proposed in \cite[Definition 3.7]{Gigli12}, but it is
rather its equivalent reformulation provided
in \cite[Proposition 3.11]{Gigli12}.
\fr}\end{remark}
\begin{lemma}
Let \((\X,\sfd,\mm)\) be a metric measure space. Let \(p,q\in(1,\infty)\)
satisfy \(\frac{1}{p}+\frac{1}{q}=1\).  Let \(\ppi\) be a \(q\)-test plan that
\(q\)-represents the gradient of some function \(f\in W^{1,p}(\X)\). Then
\begin{subequations}
\begin{align}\label{eq:tp_repr_grad_impl1}
\frac{f\circ\e_t-f\circ\e_0}{t}&\to|Df|_p^p\circ\e_0
\quad\text{ strongly in }L^1(\ppi)\text{ as }t\searrow 0,\\
\label{eq:tp_repr_grad_impl2}
\frac{{\rm E}_{q,t}}{t}&\to|Df|_p^{p/q}\circ\e_0
\quad\text{ strongly in }L^q(\ppi)\text{ as }t\searrow 0.
\end{align}
\end{subequations}
\end{lemma}
\begin{proof}
First of all, let us prove \eqref{eq:tp_repr_grad_impl2}. Let \(t_i\searrow 0\)
be fixed. Since \(\big({\rm E}_{q,t_i}/t_i\big)^{q/p}\to|Df|_p\circ\e_0\)
strongly in \(L^p(\ppi)\) as \(i\to\infty\) by \eqref{eq:tp_repr_grad_claim2},
we can assume (possibly passing to a subsequence) that
\({\rm E}_{q,t_i}/t_i\to|Df|_p^{p/q}\circ\e_0\) pointwise \(\ppi\)-a.e.\ as
\(i\to\infty\) and that there exists \(H\in L^p(\ppi)\) such that
\(\big({\rm E}_{q,t_i}/t_i\big)^{q/p}\leq H\) holds \(\ppi\)-a.e.\ for
every \(i\in\N\). In particular, for any \(i\in\N\) we have that
\[
\bigg|\frac{{\rm E}_{q,t_i}}{t_i}-|Df|_p^{p/q}\circ\e_0\bigg|^q\leq
2^{q-1}\bigg(\frac{{\rm E}_{q,t_i}}{t_i}\bigg)^q+2^{q-1}|Df|_p^p\circ\e_0\leq
2^{q-1}\big(H^p+|Df|_p^p\circ\e_0\big)\quad\ppi\text{-a.e..}
\]
Therefore, by dominated convergence theorem we get that
\(\int\big|{\rm E}_{q,t_i}/t_i-|Df|_p^{p/q}\circ\e_0\big|^q\,\d\ppi\to 0\)
as \(i\to\infty\), whence the claimed property \eqref{eq:tp_repr_grad_impl2}
follows (thanks to the arbitrariness of \(t_i\searrow 0\)).

In order to prove \eqref{eq:tp_repr_grad_impl1}, observe that
\eqref{eq:tp_repr_grad_claim1}, \eqref{eq:tp_repr_grad_impl2},
and H\"{o}lder's inequality yield
\[
\frac{f\circ\e_t-f\circ\e_0}{t}=\frac{f\circ\e_t-f\circ\e_0}{{\rm E}_{q,t}}\,
\frac{{\rm E}_{q,t}}{t}\to|Df|_p^{1+\frac{p}{q}}\circ\e_0
=|Df|_p^p\circ\e_0
\]
strongly in \(L^1(\ppi)\) as \(t\searrow 0\). The proof of the
statement is thus achieved.
\end{proof}
The existence of plans representing a gradient
has been proven in \cite[Theorem 3.14]{Gigli12}:
\begin{theorem}[Existence of test plans representing a gradient \cite{Gigli12}]
\label{thm:ex_plan_repr_grad}
Let \((\X,\sfd,\mm)\) be a metric measure space. Let \(p,q\in(1,\infty)\)
satisfy \(\frac{1}{p}+\frac{1}{q}=1\). Fix any \(\mu\in\mathscr P_q(\X)\)
such that \(\mu\leq C\mm\) for some constant \(C>0\). Then for any
\(f\in W^{1,p}(\X)\) there exists a \(q\)-test plan \(\ppi\) that
\(q\)-represents the gradient of \(f\) and satisfies \((\e_0)_\#\ppi=\mu\).
\end{theorem}
\subsubsection{Velocity of a test plan}
Another useful tool is the \textbf{velocity of a test plan} \(\ppi\),
which consists of an abstract way to define -- in a suitable sense
-- the velocity \(\gamma'_t\) at time \(t\) of \(\ppi\)-a.e.\ curve
\(\gamma\). Here, the notion of pullback of a normed module enters into play.
\begin{theorem}[Velocity of a test plan \cite{Gigli14}]\label{thm:vel_tp}
Let \((\X,\sfd,\mm)\) be a metric measure space and fix exponents
\(p,q\in(1,\infty)\) such that \(\frac{1}{p}+\frac{1}{q}=1\).
Suppose the Sobolev space \(W^{1,p}(\X)\) is separable. Let \(\ppi\) be a given
\(q\)-test plan on \((\X,\sfd,\mm)\). Then there exists a (unique up to
\(\mathcal L^1\)-a.e.\ equality) family \(\{\ppi'_t\}_{t\in[0,1]}\)
of elements \(\ppi'_t\in\big(\e_t^*L^p(T^*\X)\big)^*\) such that the following
property is satisfied: given any function \(f\in W^{1,p}(\X)\), it holds that
\[
\frac{\d}{\d t}\,f\circ\e_t\coloneqq
\lim_{h\to 0}\frac{f\circ\e_{t+h}-f\circ\e_t}{h}=\ppi'_t(\e_t^*\d_p f)
\quad\text{ for }\mathcal L^1\text{-a.e.\ }t\in[0,1],
\]
where the derivative is taken with respect to the strong topology
of \(L^1(\ppi)\). Moreover, it holds
\[
|\ppi'_t|(\gamma)=|\dot\gamma_t|\quad\text{ for }
(\ppi\otimes\mathcal L^1)\text{-a.e.\ }
(\gamma,t)\in AC^q\big([0,1],\X\big)\times[0,1].
\]
\end{theorem}
\begin{remark}{\rm
As we are going to explain, Theorem \ref{thm:vel_tp} can be proven by
repeating almost verbatim the proof of \cite[Theorem 2.3.18]{Gigli14};
the argument to deal with the case \(p=2\) can be easily adapted to treat
general exponents \(p\in(1,\infty)\). First of all, the use of the
approximation by Lipschitz functions can be avoided by arguing as in
\cite[Theorem 4.4.7]{GP20}. Also, in order to prove existence of 
the elements \(\ppi'_t\in\big(\e_t^*L^p(T^*\X)\big)^*\) just the separability
of \(W^{1,p}(\X)\) is needed; in \cite[Theorem 2.3.18]{Gigli14}
the separability of \(L^q(T\X)\), which
implies that of \(W^{1,p}(\X)\), is used to ensure that
\(\e_t^*L^q(T\X)\) can be identified with \(\big(\e_t^*L^p(T^*\X)\big)^*\)
(according to \cite[Theorem 1.6.7]{Gigli14}), while in general the former
is only isometrically embedded into the latter.
\fr}\end{remark}
\begin{proposition}\label{prop:f_et_AC}
Let \((\X,\sfd,\mm)\) be a metric measure space and \(p,q\in(1,\infty)\)
satisfy \(\frac{1}{p}+\frac{1}{q}=1\). Suppose the Sobolev space \(W^{1,p}(\X)\) is separable. Let \(\ppi\) be a \(q\)-test plan on \(\X\).
Then for every function \(f\in W^{1,p}(\X)\) it holds that the curve
\(t\mapsto f\circ\e_t\) belongs to \(AC^q\big([0,1],L^1(\ppi)\big)\) and
\begin{equation}\label{eq:f_et_AC_claim}
f\circ\e_t-f\circ\e_s=\int_s^t\ppi'_r(\e_r^*\d_p f)\,\d r
\quad\text{ for every }s,t\in[0,1]\text{ with }s<t.
\end{equation}
\end{proposition}
\begin{proof}
Let us define
\[
\phi(r)\coloneqq\bigg(\int|\dot\gamma_r|^q\,\d\ppi(\gamma)\bigg)
^{\frac{1}{q}}\quad\text{ for }\mathcal L^1\text{-a.e.\ }r\in[0,1].
\]
Given that \(\int_0^1\phi(r)^q\,\d r=\int\!\!\int_0^1|\dot\gamma_r|^q
\,\d r\,\d\ppi(\gamma)<+\infty\), we have \(\phi\in L^q(0,1)\).
Fix \(f\in W^{1,p}(\X)\) and \(s,t\in[0,1]\) with \(s<t\). It holds that
\begin{align*}
\big\|f\circ\e_t-f\circ\e_s\big\|_{L^1(\sppi)}&=
\int\big|f(\gamma_t)-f(\gamma_s)\big|\,\d\ppi(\gamma)\leq
\int\!\!\!\int_s^t|Df|_p(\gamma_r)\,|\dot\gamma_r|\,\d r\,\d\ppi(\gamma)\\
&\leq\int_s^t\bigg(\int|Df|_p^p\circ\e_r\,\d\ppi\bigg)^{\frac{1}{p}}
\bigg(\int|\dot\gamma_r|^q\,\d\ppi(\gamma)\bigg)^{\frac{1}{q}}\d r\\
&\leq{\rm Comp}(\ppi)^{\frac{1}{p}}\,\big\||Df|_p\big\|_{L^p(\mm)}
\int_s^t\phi(r)\,\d r,
\end{align*}
which shows that the curve \([0,1]\ni t\mapsto f\circ\e_t\in L^1(\ppi)\)
is \(q\)-absolutely continuous. Moreover, we know from
Theorem \ref{thm:vel_tp} that the \(L^1(\ppi)\)-derivative
\(\frac{\d}{\d t}f\circ\e_t\) exists and equals
\(\ppi'_t(\e_t^*\d_p f)\) at \(\mathcal L^1\)-a.e.\ \(t\in[0,1]\).
Therefore, the identity in \eqref{eq:f_et_AC_claim} follows from
\cite[Proposition 1.3.16]{GP20}.
\end{proof}
\section{Master test plans on metric measure spaces}
\subsection{Properties of plans representing a gradient}
In order to prove our main theorem, we first need to study some
properties of plans representing a gradient. Roughly speaking,
we aim to show that if \(\ppi\) represents the gradient of \(f\), then
for any Sobolev function \(g\) the derivative of \(t\mapsto g\circ\e_t\)
at \(t=0\) coincides with \(\langle\nabla g,\nabla f\rangle\circ\e_0\),
in a sense; see Proposition \ref{prop:conv_sc_prod}.
\begin{lemma}\label{lem:bound_above_scal_pr}
Let \((\X,\sfd,\mm)\) be a metric measure space. Let \(p,q\in(1,\infty)\)
satisfy \(\frac{1}{p}+\frac{1}{q}=1\). Suppose the Sobolev space
\(W^{1,p}(\X)\) is separable. Let \(f\in W^{1,p}(\X)\)
be given. Let \(\ppi\) be a \(q\)-test plan that \(q\)-represents the gradient
of \(f\). Then for every function \(G\in L^p(\mm)\) with \(G\geq 0\) there
exists a family \(\{\Phi_t\}_{t\in(0,1)}\subseteq L^1(\ppi)\) such that
\begin{equation}\label{eq:bound_above_scal_pr_claim}
\fint_0^t G\circ\e_s\,|\ppi'_s|\,\d s\leq\Phi_t\;\;\;\ppi\text{-a.e.}
\quad\text{ for every }t\in(0,1)
\end{equation}
and \(\Phi_t\to G\circ\e_0\,|Df|_p^{p/q}\circ\e_0\) strongly in \(L^1(\ppi)\)
as \(t\searrow 0\).
\end{lemma}
\begin{proof}
Let \(G\in L^p(\mm)\), \(G\geq 0\) be fixed.
Calling \(R_t\coloneqq\fint_0^t\big|G\circ\e_s-G\circ\e_0\big||\ppi'_s|\,\d s\),
it holds that
\[
\fint_0^t G\circ\e_s\,|\ppi'_s|\,\d s\leq
R_t+G\circ\e_0\fint_0^t|\ppi'_s|\,\d s
\leq R_t+G\circ\e_0\bigg(\fint_0^t|\ppi'_s|^q\,\d s\bigg)^{\frac{1}{q}}
\eqqcolon\Phi_t\quad\ppi\text{-a.e..}
\]
Observe that
\begin{align*}
\int R_t\,\d\ppi&=
\int\!\!\!\fint_0^t\big|G\circ\e_s-G\circ\e_0\big||\ppi'_s|\,\d s\,\d\ppi\\&\leq
\bigg(\int\!\!\!\fint_0^t\big|G\circ\e_s-G\circ\e_0\big|^p\,\d s\,\d\ppi\bigg)
^{\frac{1}{p}}
\bigg(\int\!\!\!\fint_0^t|\ppi'_s|^q\,\d s\,\d\ppi\bigg)^{\frac{1}{q}}\\
&=\bigg(\fint_0^t\big\|G\circ\e_s-G\circ\e_0\big\|_{L^p(\sppi)}^p\,\d s
\bigg)^{\frac{1}{p}}
\bigg(\int\frac{{\rm E}_{q,t}^q}{t^q}\,\d\ppi\bigg)^{\frac{1}{q}}
\to 0\quad\text{ as }t\searrow 0,
\end{align*}
where we used the fact that
\(\int{\rm E}_{q,t}^q/t^q\,\d\ppi\to\int|Df|_p^p\circ\e_0\,\d\ppi\)
as \(t\searrow 0\) and
the continuity of the mapping \([0,1]\ni s\mapsto G\circ\e_s\in L^p(\ppi)\).
Also, we have \(\big(\fint_0^t|\ppi'_s|^q\,\d s\big)^{1/q}
={\rm E}_{q,t}/t\to|Df|_p^{p/q}\circ\e_0\) strongly in \(L^q(\ppi)\) as \(t\searrow 0\),
whence accordingly \(G\circ\e_0\big(\fint_0^t|\ppi'_s|^q\,\d s\big)^{1/q}
\to G\circ\e_0\,|Df|_p^{p/q}\circ\e_0\) strongly in \(L^1(\ppi)\)
as \(t\searrow 0\).
All in all, we proved that \(\Phi_t\to G\circ\e_0\,|Df|_p^{p/q}\circ\e_0\)
in \(L^1(\ppi)\).
\end{proof}
\begin{corollary}\label{cor:aux_countable_tp}
Let \((\X,\sfd,\mm)\) be a metric measure space. Let \(p,q\in(1,\infty)\)
satisfy \(\frac{1}{p}+\frac{1}{q}=1\). Suppose the Sobolev space
\(W^{1,p}(\X)\) is separable. Let \(f\in W^{1,p}(\X)\) be given.
Let \(\ppi\) be a \(q\)-test plan that \(q\)-represents the gradient of \(f\).
Fix any \(g\in W^{1,p}(\X)\) and a sequence \(t_i\searrow 0\).
Then there exist a subsequence \((t_{i_j})_j\) and a function
\(\ell\in L^1(\ppi)\) such that
\begin{equation}\label{eq:aux_countable_tp}
\fint_0^{t_{i_j}}\ppi'_s(\e_s^*\d_p g)\,\d s\rightharpoonup\ell
\quad\text{ weakly in }L^1(\ppi)\text{ as }j\to\infty
\end{equation}
and \(|\ell|\leq|Df|_p^{p/q}\circ\e_0\,|Dg|_p\circ\e_0\)
in the \(\ppi\)-a.e.\ sense.
\end{corollary}
\begin{proof}
Pick functions \(\{\Phi_t\}_{t\in(0,1)}\subseteq L^1(\ppi)\) associated
with \(G\coloneqq|Dg|_p\) as in Lemma \ref{lem:bound_above_scal_pr}.
Given that the sequence \((\Phi_{t_i})_i\) is strongly convergent
in \(L^1(\ppi)\), we can find a subsequence
\((t_{i_j})_j\) and a non-negative function \(H\in L^1(\ppi)\) such that
\(\Phi_{t_{i_j}}\leq H\) holds \(\ppi\)-a.e.\ for every \(j\in\N\). Then
\[
\bigg|\fint_0^{t_{i_j}}\ppi'_s(\e_s^*\d_p g)\,\d s\bigg|\leq
\fint_0^{t_{i_j}}|Dg|_p\circ\e_s\,|\ppi'_s|\,\d s
\overset{\eqref{eq:bound_above_scal_pr_claim}}\leq
\Phi_{t_{i_j}}\leq H\;\;\;\ppi\text{-a.e.}\quad\text{ for every }j\in\N.
\]
Therefore, thanks to \cite[Lemma 1.3.22]{GP20} we know that there exists
a function \(\ell\in L^1(\ppi)\) such that (possibly passing to a not
relabelled subsequence) the property in \eqref{eq:aux_countable_tp}
holds. Finally, since
\(\fint_0^{t_{i_j}}\ppi'_s(\e_s^*\d_p g)\,\d s\leq\Phi_{t_{i_j}}\)
holds \(\ppi\)-a.e.\ for every \(j\in\N\) and
\(\Phi_{t_{i_j}}\to|Df|_p^{p/q}\circ\e_0\,|Dg|_p\circ\e_0\) in \(L^1(\ppi)\),
we obtain the \(\ppi\)-a.e.\ inequality
\(|\ell|\leq|Df|_p^{p/q}\circ\e_0\,|Dg|_p\circ\e_0\), getting the statement.
\end{proof}
\begin{proposition}\label{prop:conv_sc_prod}
Let \((\X,\sfd,\mm)\) be a metric measure space. Let \(p,q\in(1,\infty)\)
satisfy \(\frac{1}{p}+\frac{1}{q}=1\).  Suppose the Sobolev space
\(W^{1,p}(\X)\) is separable. Fix \(f\in W^{1,p}(\X)\). Let \(\ppi\)
be a \(q\)-test plan that \(q\)-represents the gradient of \(f\). Then for any
sequence \(t_i\searrow 0\) there exist a subsequence \((t_{i_j})_j\) and an
element \(\eta\in{\sf Dual}(\e_0^*\d_p f)\), where the mapping
\(\sf Dual\) is defined as in \eqref{eq:def_Dual}, such that
\begin{equation}\label{eq:conv_sc_prod_claim}
\frac{g\circ\e_{t_{i_j}}-g\circ\e_0}{t_{i_j}}\rightharpoonup
\eta(\e_0^*\d_p g)\quad\text{ weakly in }L^1(\ppi)\text{ as }j\to\infty,
\text{ for every }g\in W^{1,p}(\X).
\end{equation}
\end{proposition}
\begin{proof} We subdivide the proof into several steps:\\
{\color{blue}\textsc{Step 1.}}
Fix any countable, strongly dense \(\mathbb Q\)-linear subspace
\(\mathcal C\) of \(W^{1,p}(\X)\). Therefore, it holds that
\(\mathcal V\coloneqq\{\e_0^*\d_p g\,:\,g\in\mathcal C\}\)
is a generating \(\mathbb Q\)-linear subspace of \(\e_0^*L^p(T^*\X)\).
Thanks to Corollary \ref{cor:aux_countable_tp} and a diagonalisation
argument, the sequence \(t_i\searrow 0\) admits a (not relabelled)
subsequence such that \(\fint_0^{t_i}\ppi'_s(\e_s^*\d_p g)\,\d s
\rightharpoonup\ell_g\) weakly in \(L^1(\ppi)\) for every \(g\in\mathcal C\),
for some limit functions \(\ell_g\in L^1(\ppi)\) satisfying
\(|\ell_g|\leq|Df|_p^{p/q}\circ\e_0\,|Dg|_p\circ\e_0\) in the
\(\ppi\)-a.e.\ sense. Let us define
\[
L\colon\mathcal V\to L^1(\ppi),\quad L(\e_0^*\d_p g)\coloneqq\ell_g
\text{ for every }\e_0^*\d_p g\in\mathcal V.
\]
Given that \(\big|L(\e_0^*\d_p g)\big|\leq|Df|_p^{p/q}\circ\e_0\,|\e_0^*\d_p g|\)
holds \(\ppi\)-a.e., we deduce that \(L\) is a well-defined, linear,
and continuous mapping. Therefore, \cite[Proposition 3.2.9]{GP20}
grants the existence of a unique element
\(\eta\in\big(\e_0^*L^p(T^*\X)\big)^*\) such that
\(\eta(\e_0^*\d_p g)=L(\e_0^*\d_p g)\) is satisfied for every \(g\in\mathcal C\)
and \(|\eta|\leq|Df|_p^{p/q}\circ\e_0=|\e_0^*\d_p f|^{p/q}\) in the
\(\ppi\)-a.e.\ sense. Accordingly, it holds that
\begin{equation}\label{eq:conv_sc_prod_aux}
\fint_0^{t_i}\ppi'_s(\e_s^*\d_p g)\,\d s\rightharpoonup\eta(\e_0^*\d_p g)
\quad\text{ weakly in }L^1(\ppi)\text{ as }i\to\infty,
\text{ for every }g\in\mathcal C.
\end{equation}
{\color{blue}\textsc{Step 2.}}
Let \(g\in W^{1,p}(\X)\) be fixed. Choose any sequence
\((g_n)_n\subseteq\mathcal C\) such that \(g_n\to g\) strongly
in \(W^{1,p}(\X)\). Fix any \(h\in L^\infty(\ppi)\) and some
constant \(M>0\) satisfying \(\int{\rm E}_{q,t_i}^q/t_i^q\,\d\ppi\leq M^q\)
for every \(i\in\N\). Given any \(i,n\in\N\), we can estimate
\begin{equation}\label{eq:conv_sc_prod_aux2}
\bigg|\int h\fint_0^{t_i}\ppi'_s(\e_s^*\d_p g)\,\d s\,\d\ppi
-\int h\,\eta(\e_0^*\d_p g)\,\d\ppi\bigg|\leq A_{i,n}+B_{i,n}+C_n,
\end{equation}
where we set
\begin{align*}
A_{i,n}&\coloneqq\bigg|\int h\fint_0^{t_i}\ppi'_s\big(\e_s^*\d_p(g-g_n)\big)
\,\d s\,\d\ppi\bigg|,\\
B_{i,n}&\coloneqq\bigg|\int h\fint_0^{t_i}\ppi'_s(\e_s^*\d_p g_n)
\,\d s\,\d\ppi-\int h\,\eta(\e_0^*\d_p g_n)\,\d\ppi\bigg|,\\
C_n&\coloneqq\bigg|\int h\,\eta\big(\e_0^*\d_p(g_n-g)\big)\,\d\ppi\bigg|.
\end{align*}
Observe that
\begin{align*}
A_{i,n}&\leq\|h\|_{L^\infty(\sppi)}\int\!\!\!\fint_0^{t_i}
\big|D(g-g_n)\big|_p\circ\e_s\,|\ppi'_s|\,\d s\,\d\ppi\\
&\leq\|h\|_{L^\infty(\sppi)}\bigg(\int\!\!\!\fint_0^{t_i}\big|D(g-g_n)\big|_p^p
\circ\e_s\,\d s\,\d\ppi\bigg)^{\frac{1}{p}}\bigg(\int\!\!\!\fint_0^{t_i}
|\ppi'_s|^q\,\d s\,\d\ppi\bigg)^{\frac{1}{q}}\\
&\leq{\rm Comp}(\ppi)^{\frac{1}{p}}\,\|h\|_{L^\infty(\sppi)}
\bigg(\int\big|D(g-g_n)\big|_p^p\,\d\mm\bigg)^{\frac{1}{p}}
\bigg(\int\frac{{\rm E}_{q,t_i}^q}{t_i^q}\,\d\ppi\bigg)^{\frac{1}{q}}\\
&\leq M\,{\rm Comp}(\ppi)^{\frac{1}{p}}\,\|h\|_{L^\infty(\sppi)}
\,\|g-g_n\|_{W^{1,p}(\X)}.
\end{align*}
Moreover, it follows from \eqref{eq:conv_sc_prod_aux} that
\(\lim_{i\to\infty}B_{i,n}=0\) for any given \(n\in\N\).
Finally, we estimate
\begin{align*}
C_n&\leq\|h\|_{L^\infty(\sppi)}\int\big|D(g_n-g)
\big|_p\circ\e_0\,|\eta|\,\d\ppi\\&\leq\|h\|_{L^\infty(\sppi)}
\bigg(\int\big|D(g_n-g)\big|_p^p\circ\e_0\,\d\ppi\bigg)^{\frac{1}{p}}
\bigg(\int|\eta|^q\,\d\ppi\bigg)^{\frac{1}{q}}\\
&\leq{\rm Comp}(\ppi)^{\frac{1}{p}}\,\|h\|_{L^\infty(\sppi)}
\bigg(\int\big|D(g_n-g)\big|_p^p\,\d\mm\bigg)^{\frac{1}{p}}
\bigg(\int|Df|_p^p\circ\e_0\,\d\ppi\bigg)^{\frac{1}{q}}\\
&\leq{\rm Comp}(\ppi)\,\|h\|_{L^\infty(\sppi)}\,
\|g_n-g\|_{W^{1,p}(\X)}\,\|f\|_{W^{1,p}(\X)}^{p/q}.
\end{align*}
Hence, given any \(\eps>0\) we can find \(n\in\N\) such that
\(A_{i,n}+C_n\leq\eps\) for every \(i\in\N\). Then
\[
\lims_{i\to\infty}\bigg|\int h\fint_0^{t_i}\ppi'_s(\e_s^*\d_p g)\,\d s\,\d\ppi
-\int h\,\eta(\e_0^*\d_p g)\,\d\ppi\bigg|
\overset{\eqref{eq:conv_sc_prod_aux2}}\leq\eps+\lim_{i\to\infty}B_{i,n}=\eps.
\]
By letting \(\eps\searrow 0\), we conclude that
\(\lim_{i\to\infty}\int h\fint_0^{t_i}\ppi'_s(\e_s^*\d_p g)\,\d s\,\d\ppi
=\int h\,\eta(\e_0^*\d_p g)\,\d\ppi\) holds for every \(h\in L^\infty(\ppi)\),
whence
\(\fint_0^{t_i}\ppi'_s(\e_s^*\d_p g)\,\d s\rightharpoonup\eta(\e_0^*\d_p g)\)
weakly in \(L^1(\ppi)\) as \(i\to\infty\). Given that we have
\((g\circ\e_{t_i}-g\circ\e_0)/t_i=\fint_0^{t_i}\ppi'_s(\e_s^*\d_p g)\,\d s\)
by Proposition \ref{prop:f_et_AC}, we have proven that
\begin{equation}\label{eq:conv_sc_prod_aux3}
\frac{g\circ\e_{t_i}-g\circ\e_0}{t_i}\rightharpoonup\eta(\e_0^*\d_p g)
\quad\text{ weakly in }L^1(\ppi)\text{ as }i\to\infty,
\text{ for every }g\in W^{1,p}(\X).
\end{equation}
{\color{blue}\textsc{Step 3.}}
We aim to show that \(\eta\in{\sf Dual}(\e_0^*\d_p f)\).
Since \(\ppi\) represents the gradient of \(f\), one has
\[
\frac{f\circ\e_{t_i}-f\circ\e_0}{t_i}\to|Df|_p^p\circ\e_0
=|\e_0^*\d_p f|^p\quad\text{ strongly in }L^1(\ppi)\text{ as }i\to\infty.
\]
Hence, by applying \eqref{eq:conv_sc_prod_aux3} with \(g\coloneqq f\) we
obtain that \(\eta(\e_0^*\d_p f)=|\e_0^*\d_p f|^p\) holds \(\ppi\)-a.e.. Then
\[
|\e_0^*\d_p f|^p=\eta(\e_0^*\d_p f)\leq|\eta||\e_0^*\d_p f|\leq
|\e_0^*\d_p f|^{\frac{p}{q}+1}=|\e_0^*\d_p f|^p\quad\ppi\text{-a.e.,}
\]
whence \(|\eta|=|\e_0^*\d_p f|^{p/q}\) holds \(\ppi\)-a.e.\ and accordingly
\(\eta\in{\sf Dual}(\e_0^*\d_p f)\), as required.
\end{proof}
Albeit not strictly needed for the purposes of this article, let us
illustrate a reinforcement of Proposition \ref{prop:conv_sc_prod}
in the case of an infinitesimally Hilbertian ambient space:
\begin{corollary}
Let \((\X,\sfd,\mm)\) be an infinitesimally Hilbertian metric measure space.
Fix any function \(f\in W^{1,2}(\X)\). Let \(\ppi\) be a \(2\)-test plan on
\(\X\) that \(2\)-represents the gradient of \(f\). Then
\[
\frac{g\circ\e_t-g\circ\e_0}{t}\rightharpoonup
\langle\nabla g,\nabla f\rangle\circ\e_0\quad
\text{ weakly in }L^1(\ppi)\text{ as }t\searrow 0,\text{ for every }
g\in W^{1,2}(\X).
\]
\end{corollary}
\begin{proof}
First, the infinitesimal Hilbertianity assumption grants that
\(W^{1,2}(\X)\) and \(L^2(T\X)\) are separable; see, \emph{e.g.},
\cite[Proposition 4.3.5]{GP20}. In particular, we know from
\cite[Theorem 1.6.7]{Gigli14} that the space \(\big(\e_0^*L^2(T^*\X)\big)^*\)
is isometrically isomorphic to \(\e_0^*L^2(T\X)\). Thanks to this fact,
we can identify any element \(\eta\) satisfying \eqref{eq:conv_sc_prod_claim}
(for some \(t_{i_j}\searrow 0\)) with an element \(v\) of the pullback module
\(\e_0^*L^2(T\X)\). Given that \((\e_0^*\d_2 f)(v)=|\e_0^*\d_2 f|^2=|v|^2\) holds
\(\ppi\)-a.e., we get
\[
|v-\e_0^*\nabla f|^2=|v|^2-2\,\langle v,\e_0^*\nabla f\rangle+|\e_0^*\nabla f|
^2=|v|^2-2\,(\e_0^*\d_2 f)(v)+|\e_0^*\d_2 f|^2=0\quad\ppi\text{-a.e.,}
\]
whence \(v=\e_0^*\nabla f\). In particular, the limit \(v\) does not depend
on \((t_{i_j})_j\), thus accordingly
\[
\frac{g\circ\e_t-g\circ\e_0}{t}\rightharpoonup
(\e_0^*\d_2 g)(\e_0^*\nabla f)=\langle\nabla g,\nabla f\rangle\circ\e_0
\quad\text{ weakly in }L^1(\ppi)\text{ as }t\searrow 0
\]
for every \(g\in W^{1,2}(\X)\). Therefore, the statement is achieved.
\end{proof}
\subsection{Existence of master test plans on metric measure spaces}
We now have at our disposal all the ingredients that we need to prove
our main theorem, which says that a single test plan is sufficient to
identify the minimal relaxed slope of every Sobolev function. In this
regard, the relevant notion is that of master test plan:
\begin{definition}[Master test plan]
Let \((\X,\sfd,\mm)\) be a metric measure space. Fix \(p,q\in(1,\infty)\)
such that \(\frac{1}{p}+\frac{1}{q}=1\). Then a \(q\)-test plan
\(\ppi_q\) on \((\X,\sfd,\mm)\) is said to be a \textbf{master \(q\)-test plan}
provided it holds that
\[
|Df|_{\sppi_q,p}=|Df|_p\quad\text{ for every }f\in W^{1,p}(\X).
\]
Here we are using the fact that \(W^{1,p}(\X)\subseteq W^{1,p}_{\sppi_q}(\X)\),
which is granted by Proposition \ref{prop:ineq_mwug}.
\end{definition}
Hence, our main result about identification of the minimal relaxed slope
reads as follows:
\begin{theorem}[Existence of master test plans]\label{thm:master_tp}
Let \((\X,\sfd,\mm)\) be a metric measure space. Fix any
\(p,q\in(1,\infty)\) such that \(\frac{1}{p}+\frac{1}{q}=1\).
Suppose the Sobolev space \(W^{1,p}(\X)\) is separable. Then there exists
a master \(q\)-test plan \(\ppi_q\) on \((\X,\sfd,\mm)\).
\end{theorem}
\begin{proof}
We subdivide the proof into several steps:\\
{\color{blue}\textsc{Step 1.}}
First of all, fix a countable family \(\mathcal C\subseteq W^{1,p}(\X)\)
that is strongly dense in \(W^{1,p}(\X)\). Fix any measure
\(\tilde\mm\in\mathscr P_q(\X)\) such that \(\mm\ll\tilde\mm\leq C\mm\)
for some \(C>0\), whose existence is shown in Remark
\ref{rmk:ex_tilde_m}. Given any \(f\in\mathcal C\),
there exists a \(q\)-test plan \(\ppi^f\) on \(\X\) that \(q\)-represents
the gradient of \(f\) and satisfies \((\e_0)_\#\ppi^f=\tilde\mm\)
(by Theorem \ref{thm:ex_plan_repr_grad}). Let us define
\(\Pi\coloneqq\{\ppi^f\,:\,f\in\mathcal C\}\). We aim to prove that
\begin{equation}\label{eq:countable_tp_claim2}
|Df|_{\Pi,p}=|Df|_p\quad\text{ for every }f\in W^{1,p}(\X).
\end{equation}
Since \(|Df|_{\Pi,p}\leq|Df|_p\) holds \(\mm\)-a.e.\ by Proposition
\ref{prop:ineq_mwug}, to prove \eqref{eq:countable_tp_claim2}
it suffices to show that
\begin{equation}\label{eq:countable_tp_aux}
\int|Df|_p^p\,\d\tilde\mm\leq\int|Df|_{\Pi,p}^p\,\d\tilde\mm
\quad\text{ for every }f\in W^{1,p}(\X).
\end{equation}
{\color{blue}\textsc{Step 2.}}
In order to show \eqref{eq:countable_tp_aux}, let \(f\in W^{1,p}(\X)\)
be fixed. Choose any sequence \((f_n)_n\subseteq\mathcal C\)
that strongly converges to \(f\) in \(W^{1,p}(\X)\). Possibly passing
to a (not relabelled) subsequence, we may assume that \(|Df_n|_p\to|Df|_p\)
pointwise \(\mm\)-a.e.\ and that there exists a function \(G\in L^p(\mm)\)
such that \(|Df_n|_p\leq G\) holds \(\mm\)-a.e.\ for every \(n\in\N\).
For brevity, let us put \(\ppi^n\coloneqq\ppi^{f_n}\)
for every \(n\in\N\). Given any \(n\in\N\), by applying Proposition
\ref{prop:conv_sc_prod} we obtain that there exist an element
\(\eta_n\in{\sf Dual}(\e_0^*\d_p f_n)\) and a sequence
\((t^n_i)_i\subseteq(0,1)\) with \(\lim_{i\to\infty} t^n_i=0\) such that
\begin{equation}\label{eq:countable_tp_aux2}
\frac{g\circ\e_{t^n_i}-g\circ\e_0}{t^n_i}\rightharpoonup
\eta_n(\e_0^*\d_p g)\quad\text{ weakly in }L^1(\ppi^n)\text{ as }
i\to\infty,\text{ for every }g\in W^{1,p}(\X).
\end{equation}
Therefore, by choosing \(g\coloneqq f\) in \eqref{eq:countable_tp_aux2}
we deduce that
\begin{align*}
\int\eta_n(\e_0^*\d_p f)\,\d\ppi^n&=
\lim_{i\to\infty}\int\frac{f\circ\e_{t^n_i}-f\circ\e_0}{t^n_i}\,\d\ppi^n
\leq\limi_{i\to\infty}\frac{1}{t^n_i}\int\big|f(\gamma_{t^n_i})-f(\gamma_0)
\big|\,\d\ppi^n(\gamma)\\
&\leq\limi_{i\to\infty}\int\!\!\!\fint_0^{t^n_i}\bigg|\frac{\d}{\d s}\,
f(\gamma_s)\bigg|\,\d s\,\d\ppi^n(\gamma)\leq
\limi_{i\to\infty}\int\!\!\!\fint_0^{t^n_i}|Df|_{\Pi,p}(\gamma_s)\,|\dot\gamma_s|
\,\d s\,\d\ppi^n(\gamma)\\
&\leq\limi_{i\to\infty}\bigg(\int\!\!\!\fint_0^{t^n_i}|Df|_{\Pi,p}^p\circ\e_s
\,\d s\,\d\ppi^n\bigg)^{\frac{1}{p}}\bigg(\int\!\!\!\fint_0^{t^n_i}
|\dot\gamma_s|^q\,\d s\,\d\ppi^n(\gamma)\bigg)^{\frac{1}{q}}\\
&=\lim_{i\to\infty}\bigg(\fint_0^{t^n_i}\big\||Df|_{\Pi,p}\circ\e_s\big\|
_{L^p(\sppi^n)}^p\,\d s\bigg)^{\frac{1}{p}}\bigg(\int
\frac{{\rm E}_{q,t^n_i}^q}{(t^n_i)^q}\,\d\ppi^n\bigg)^{\frac{1}{q}}\\
&=\bigg(\int|Df|_{\Pi,p}^p\circ\e_0\,\d\ppi^n\bigg)^{\frac{1}{p}}
\bigg(\int|Df_n|_p^p\circ\e_0\,\d\ppi^n\bigg)^{\frac{1}{q}}\\
&=\big\||Df|_{\Pi,p}\big\|_{L^p(\tilde\mm)}\,\big\||Df_n|_p\big\|
_{L^p(\tilde\mm)}^{p/q}.
\end{align*}
Furthermore, observe that for any \(n\in\N\) it holds that
\begin{align*}
&\bigg|\int|Df|_p^p\,\d\tilde\mm-\int\eta_n(\e_0^*\d_p f)\,\d\ppi^n\bigg|\\
\leq\,&\bigg|\int|Df|_p^p\,\d\tilde\mm-\int\eta_n(\e_0^*\d_p f_n)\,\d\ppi^n\bigg|
+\bigg|\int\eta_n\big(\e_0^*\d_p(f_n-f)\big)\,\d\ppi^n\bigg|\\
\leq\,&\bigg|\int|Df|_p^p\,\d\tilde\mm-\int|\e_0^*\d_p f_n|^p\,\d\ppi^n\bigg|
+\int|\eta_n|\big|D(f_n-f)\big|_p\circ\e_0\,\d\ppi^n\\
\leq\,&\bigg|\int|Df|_p^p\,\d\tilde\mm-\int|Df_n|_p^p\,\d\tilde\mm\bigg|
+\bigg(\int|\eta_n|^q\,\d\ppi^n\bigg)^{\frac{1}{q}}
\bigg(\int\big|D(f_n-f)\big|_p^p\,\d\tilde\mm\bigg)^{\frac{1}{p}}\\
\leq\,&\bigg|\int|Df|_p^p\,\d\tilde\mm-\int|Df_n|_p^p\,\d\tilde\mm\bigg|
+C^{\frac{1}{p}}\bigg(\int|Df_n|_p^p\,\d\tilde\mm\bigg)^{\frac{1}{q}}
\|f_n-f\|_{W^{1,p}(\X)}.
\end{align*}
Since \(|Df_n|_p^p\to|Df|_p^p\) pointwise \(\tilde\mm\)-a.e.\ and
\(|Df_n|_p^p\leq G^p\in L^1(\tilde\mm)\) holds \(\tilde\mm\)-a.e.\ for
all \(n\in\N\), by using the dominated convergence theorem we obtain
that \(\int|D f_n|_p^p\,\d\tilde\mm\to\int|Df|_p^p\,\d\tilde\mm\).
Consequently, by letting \(n\to\infty\) in the above estimates we get
\(\int\eta_n(\e_0^*\d_p f)\,\d\ppi^n\to\int|Df|_p^p\,\d\tilde\mm\)
as \(n\to\infty\). All in all, we can conclude that
\begin{align*}
\int|Df|_p^p\,\d\tilde\mm&=\lim_{n\to\infty}\int\eta_n(\e_0^*\d_p f)\,\d\ppi^n
\leq\big\||Df|_{\Pi,p}\big\|_{L^p(\tilde\mm)}\lim_{n\to\infty}
\big\||Df_n|_p\big\|_{L^p(\tilde\mm)}^{p/q}\\
&\leq\big\||Df|_{\Pi,p}\big\|_{L^p(\tilde\mm)}\,
\big\||Df|_p\big\|_{L^p(\tilde\mm)}^{p/q}.
\end{align*}
This proves the validity of \eqref{eq:countable_tp_aux}
and accordingly of \eqref{eq:countable_tp_claim2}.\\
{\color{blue}\textsc{Step 3.}}
Finally, it remains to show how to get the claim
from \eqref{eq:countable_tp_claim2}. Call \(\Pi=(\ppi^k)_k\) and set
\[
\eeta\coloneqq\sum_{k=1}^\infty\frac{\ppi^k}{2^k\max\big\{
{\rm Comp}(\ppi^k),{\rm KE}_q(\ppi^k),1\big\}},\qquad
\ppi_q\coloneqq\frac{\eeta}{\eeta\big(C\big([0,1],\X\big)\big)}.
\]
Since all measures \(\ppi^k\) are Borel measures concentrated
on \(AC^q\big([0,1],\X\big)\), we have that \(\eeta\) is a Borel
measure concentrated on \(AC^q\big([0,1],\X\big)\) as well. Also,
\(\eeta\big(C\big([0,1],\X\big)\big)\leq\sum_{k=1}^\infty 1/2^k=1\),
so that \(\ppi_q\) is well-defined and thus a Borel probability
measure concentrated on \(AC^q\big([0,1],\X\big)\). Given any
\(t\in[0,1]\) and a Borel set \(E\subseteq\X\), we have that
\[
(\e_t)_\#\eeta(E)=\eeta\big(\e_t^{-1}(E)\big)\leq\sum_{k=1}^\infty
\frac{\ppi^k\big(\e_t^{-1}(E)\big)}{2^k\,{\rm Comp}(\ppi^k)}
\leq\mm(E)\sum_{k=1}^\infty\frac{1}{2^k}=\mm(E),
\]
whence \(\ppi_q\) satisfies the item i) of Definition \ref{def:test_plan}.
Moreover, observe that
\[
\int\!\!\!\int_0^1|\dot\gamma_t|^q\,\d t\,\d\eeta(\gamma)
\leq\sum_{\substack{k\in\N:\\{\rm KE}_q(\sppi^k)>0}}\frac{1}
{2^k\,{\rm KE}_q(\ppi^k)}\int\!\!\!\int_0^1|\dot\gamma_t|^q\,\d t
\,\d\ppi^k(\gamma)\leq\sum_{k=1}^\infty\frac{1}{2^k}=1,
\]
thus accordingly \(\ppi_q\) has finite kinetic \(q\)-energy.
All in all, \(\ppi_q\) is a \(q\)-test plan on \((\X,\sfd,\mm)\).

Finally, a given Borel subset of \(C\big([0,1],\X\big)\)
is \(\ppi_q\)-negligible if and only if it is \(\ppi^k\)-negligible for
all \(k\in\N\), thus \(W^{1,p}_{\sppi_q}(\X)=W^{1,p}_\Pi(\X)\)
and \(|Df|_{\sppi_q,p}=|Df|_{\Pi,p}\) holds for every
\(f\in W^{1,p}_{\sppi_q}(\X)\).
Consequently, the statement follows from \eqref{eq:countable_tp_claim2}.
\end{proof}
\begin{problem}\label{pb:open_pb}
Under the assumption of Theorem \ref{thm:master_tp}, does it hold
that \(W^{1,p}_{\sppi_q}(\X)=W^{1,p}(\X)\)? In other words, is the
\(q\)-test plan \(\ppi_q\) sufficient to detect which functions
are Sobolev, and not only to identify the minimal \(p\)-relaxed slope
of those functions that are known to be Sobolev?
\end{problem}

A positive answer to the above question is known, for instance,
in the Euclidean space (and, similarly, on Riemannian manifolds).
Indeed, in this case the original approach to weakly differentiable
functions pioneered by B.\ Levi \cite{Levi06} shows that to look
at the behaviour along coordinate directions is sufficient to
distinguish the Sobolev functions; by building upon this result,
one can find a master \(q\)-test plan on \(\R^n\) for
which \(W^{1,p}_{\sppi_q}(\R^n)=W^{1,p}(\R^n)\).
\section{Master test plans on \texorpdfstring{\(\sf RCD\)}{RCD} spaces}
Aim of this section is to improve Theorem \ref{thm:master_tp} (when \(p=2\))
in the case in which the space \((\X,\sfd,\mm)\) under consideration
is a \textbf{\({\sf RCD}(K,\infty)\) space} for some \(K\in\R\).
A \({\sf RCD}(K,\infty)\) space is an infinitesimally
Hilbertian space  whose Ricci curvature is bounded from below by \(K\),
in a synthetic sense. For an account on this theory, we refer to
\cite{AmbrosioICM} and the references therein.
\medskip

An important feature of \({\sf RCD}(K,\infty)\) spaces is the presence
of a vast class of `highly regular' functions, which are referred to
as the \textbf{test functions}. In order to introduce them, we first
need to recall the notion of \textbf{Laplacian}: we declare that
\(f\in W^{1,2}(\X)\) belongs to \(D(\Delta)\) provided there
exists a (uniquely determined) function \(\Delta f\in L^2(\mm)\) such that
\[
\int g\,\Delta f\,\d\mm=-\int\langle\nabla g,\nabla f\rangle\,\d\mm
\quad\text{ for every }g\in W^{1,2}(\X).
\]
With this said, we are in a position to define
\[
{\rm Test}^\infty(\X)\coloneqq\Big\{f\in D(\Delta)\cap L^\infty(\mm)\;\Big|
\;|Df|_2\in L^\infty(\mm),\,\Delta f\in W^{1,2}(\X)\cap L^\infty(\mm)\Big\}.
\]
As proven in \cite{Savare13,Gigli14}, the family \({\rm Test}^\infty(\X)\)
is strongly dense in the Sobolev space \(W^{1,2}(\X)\).
\subsection{Regular Lagrangian flow}
Another important ingredient that we will need to prove Theorem
\ref{thm:master_tp_RCD} is the notion of regular Lagrangian
flow, which (in the metric setting) has been introduced by
L.\ Ambrosio and D.\ Trevisan in \cite{Ambrosio-Trevisan14}.
The following result is only a very special case of a much more
general statement, but still it is sufficient for our purposes;
the formulation is taken from \cite{Gigli17}.
\begin{theorem}[Regular Lagrangian flow \cite{Ambrosio-Trevisan14}]
\label{thm:RLF}
Let \((\X,\sfd,\mm)\) be a \({\sf RCD}(K,\infty)\) space, for some
constant \(K\in\R\). Let \(f\in{\rm Test}^\infty(\X)\) be given.
Then there exists a (\(\mm\)-a.e.\ uniquely determined)
\textbf{regular Lagrangian flow}
\(F_\cdot\colon\X\to C\big([0,1],\X\big)\) associated with \(\nabla f\),
which means that:
\begin{itemize}
\item[\(\rm i)\)] The map \(F_\cdot\colon\X\to C\big([0,1],\X\big)\)
is Borel and satisfies \(F_0(x)=x\) for \(\mm\)-a.e.\ \(x\in\X\).
\item[\(\rm ii)\)] There exists a constant \(L>0\) such that
\((F_t)_\#\mm\leq L\mm\) for every \(t\in[0,1]\).
\item[\(\rm iii)\)] Given any function \(g\in W^{1,2}(\X)\),
it holds that \([0,1]\ni t\mapsto g\big(F_t(x)\big)\) belongs
to the space \(W^{1,1}(0,1)\) for \(\mm\)-a.e.\ \(x\in\X\) and satisfies
\[
\frac{\d}{\d t}\,g\big(F_t(x)\big)
=\langle\nabla g,\nabla f\rangle\big(F_t(x)\big)\quad
\text{ for }(\mm\otimes\mathcal L^1)\text{-a.e.\ }(x,t)\in\X\times[0,1].
\]
\end{itemize}
\end{theorem}
\begin{remark}{\rm
Observe that item ii) is meaningful since the map
\([0,1]\times\X\ni(t,x)\mapsto F_t(x)\in\X\) is Borel
(as it is a Carath\'{e}odory function), thus in particular
\(\X\ni x\mapsto F_t(x)\in\X\) is Borel for every \(t\in[0,1]\).
Moreover, item iii) is well-posed thanks to item ii): given that
\(\langle\nabla g,\nabla f\rangle\) is defined \(\mm\)-a.e.\ and
\((F_t)_\#\mm\ll\mm\), we have that
\(\langle\nabla g,\nabla f\rangle\circ F_t\) is defined
\(\mm\)-a.e.\ as well.
\fr}\end{remark}

Given any measure \(\mu\in\mathscr P(\X)\) such that \(\mu\leq C\mm\)
for some constant \(C>0\), it holds that
\begin{equation}\label{eq:tp_associated_RLF}
\ppi\coloneqq(F_\cdot)_\#\mu\quad\text{ is a }\infty\text{-test plan on }\X.
\end{equation}
Also, we have that \(\ppi'_t=\e_t^*\nabla f\) holds for
\(\mathcal L^1\)-a.e.\ \(t\in[0,1]\). We refer to \cite{Gigli17}
for more details.
\subsection{Existence of master test plans on
\texorpdfstring{\(\sf RCD\)}{RCD} spaces}
To begin with, we show that the regularity result in 
Proposition \ref{prop:f_et_AC} can be sharpened when
the test plan is induced by a regular Lagrangian flow
(in the sense of \eqref{eq:tp_associated_RLF} above):
\begin{lemma}\label{lem:f_et_C1}
Let \((\X,\sfd,\mm)\) be a \({\sf RCD}(K,\infty)\) space, for some
\(K\in\R\). Let \(f\in{\rm Test}^\infty(\X)\) be given.
Denote by \(F_\cdot\) the regular Lagrangian flow associated with
\(\nabla f\). Let \(\mu\in\mathscr P(\X)\) be such that \(\mu\leq C\mm\)
for some \(C>0\) and define \(\ppi\coloneqq(F_\cdot)_\#\mu\). Then
for any \(g\in W^{1,2}(\X)\) it holds that the map
\([0,1]\ni t\mapsto g\circ\e_t\in L^1(\ppi)\) is of class \(C^1\) and
\[
\frac{\d}{\d t}\,g\circ\e_t=\langle\nabla g,\nabla f\rangle\circ\e_t
\quad\text{ for every }t\in[0,1].
\]
\end{lemma}
\begin{proof}
We know from Proposition \ref{prop:f_et_AC} that the curve
\([0,1]\ni t\mapsto g\circ\e_t\in L^1(\ppi)\) is absolutely
continuous and its \(L^1(\ppi)\)-strong derivative coincides
with \(D_t\coloneqq\ppi'_t(\e_t^*\d_2 g)=
\langle\nabla g,\nabla f\rangle\circ\e_t\)
for \(\mathcal L^1\)-a.e.\ \(t\in[0,1]\). Since the curve
\([0,1]\ni t\mapsto D_t\in L^1(\ppi)\) is continuous by
Proposition \ref{prop:cont_f_et}, the statement follows.
\end{proof}

We are now in a position to prove our existence result.
Even though the ideas are very similar to those carried out
in the proof of Theorem \ref{thm:master_tp}, we still prefer to
write down the whole argument since it presents many
technical simplifications.
\begin{theorem}[Master test plans on \(\sf RCD\) spaces]
\label{thm:master_tp_RCD}
Let \((\X,\sfd,\mm)\) be a \({\sf RCD}(K,\infty)\) space, for some
\(K\in\R\). Then there exists a \(\infty\)-test plan \(\ppi_2\) on
\((\X,\sfd,\mm)\) that is a master \(2\)-test plan.
\end{theorem}
\begin{proof}
Given that \((\X,\sfd,\mm)\) is infinitesimally Hilbertian, we know
from \cite[Proposition 4.3.5]{GP20} that \(W^{1,2}(\X)\) is separable,
thus we can find a countable family
\(\mathcal C\subseteq{\rm Test}^\infty(\X)\) that is strongly dense
in \(W^{1,2}(\X)\). Choose any \(\tilde\mm\in\mathscr P_2(\X)\) such
that \(\mm\ll\tilde\mm\leq C\mm\) for some \(C>0\) (recall Remark
\ref{rmk:ex_tilde_m}). Given any \(f\in\mathcal C\), we call
\(F^f_\cdot\) the regular Lagrangian flow associated with \(\nabla f\)
and we set \(\ppi^f\coloneqq(F^f_\cdot)_\#\tilde\mm\). Let us then
define \(\Pi\coloneqq\{\ppi^f\,:\,f\in\mathcal C\}\). We claim that
\begin{equation}\label{eq:master_tp_RCD_aux}
|Df|_{\Pi,2}=|Df|_2\quad\text{ for every }f\in W^{1,2}(\X).
\end{equation}
Given that \(|Df|_{\Pi,2}\leq|Df|_2\) holds \(\mm\)-a.e.\ by
Proposition \ref{prop:ineq_mwug}, it is just sufficient to show the
inequality \(\int|Df|_2^2\,\d\tilde\mm\leq\int|Df|_{\Pi,2}^2\,\d\tilde\mm\).
To this aim, fix a sequence \((f_n)_n\subseteq\mathcal C\) with
\(f_n\to f\) strongly in \(W^{1,2}(\X)\) and \(|Df_n|_2\to|Df|_2\)
strongly in \(L^2(\tilde\mm)\). Call \(\ppi^n\coloneqq\ppi^{f_n}\)
for every \(n\in\N\). Note that
\((\e_0)_\#\ppi^n=(F^{f_n}_0)_\#\tilde\mm=\tilde\mm\), so
Lemma \ref{lem:f_et_C1} and dominated convergence theorem yield
\begin{align*}
\int|Df|_2^2\,\d\tilde\mm&=
\lim_{n\to\infty}\int\langle\nabla f,\nabla f_n\rangle\,\d\tilde\mm=
\lim_{n\to\infty}\int\langle\nabla f,\nabla f_n\rangle\circ\e_0\,\d\ppi^n\\
&=\lim_{n\to\infty}\lim_{t\searrow 0}\int\frac{f\circ\e_t-f\circ\e_0}{t}\,
\d\ppi^n\leq\limi_{n\to\infty}\limi_{t\searrow 0}\int\frac{\big|f(\gamma_t)
-f(\gamma_0)\big|}{t}\,\d\ppi^n(\gamma)\\
&\leq\limi_{n\to\infty}\limi_{t\searrow 0}\int\!\!\!\fint_0^t
|Df|_{\Pi,2}(\gamma_s)\,|\dot\gamma_s|\,\d s\,\d\ppi^n(\gamma)\\
&\leq\lim_{n\to\infty}\lim_{t\searrow 0}\bigg(\int\!\!\!\fint_0^t
|Df|_{\Pi,2}^2\circ\e_s\,\d\ppi^n\,\d s\bigg)^{\frac{1}{2}}
\bigg(\int\!\!\!\fint_0^t\big|(\ppi^n)'_s\big|^2\,\d\ppi^n\,\d s\bigg)
^{\frac{1}{2}}\\
&=\lim_{n\to\infty}\lim_{t\searrow 0}\bigg(\fint_0^t\big\|
|Df|_{\Pi,2}\circ\e_s\big\|_{L^2(\sppi^n)}^2\,\d s\bigg)^{\frac{1}{2}}
\bigg(\fint_0^t\big\||Df_n|_2\circ\e_s\big\|_{L^2(\sppi^n)}^2\,\d s\bigg)
^{\frac{1}{2}}\\
&=\lim_{n\to\infty}\bigg(\int|Df|_{\Pi,2}^2\circ\e_0\,\d\ppi^n\bigg)
^{\frac{1}{2}}\bigg(\int|Df_n|_2^2\circ\e_0\,\d\ppi^n\bigg)^{\frac{1}{2}}\\
&=\bigg(\int|Df|_{\Pi,2}^2\,\d\tilde\mm\bigg)^{\frac{1}{2}}
\lim_{n\to\infty}\bigg(\int|Df_n|_2^2\,\d\tilde\mm\bigg)^{\frac{1}{2}}\\
&=\bigg(\int|Df|_{\Pi,2}^2\,\d\tilde\mm\bigg)^{\frac{1}{2}}
\bigg(\int|Df|_2^2\,\d\tilde\mm\bigg)^{\frac{1}{2}}.
\end{align*}
Therefore, the claimed identity \eqref{eq:master_tp_RCD_aux} is satisfied.
In order to conclude, it remains to pass from the countable family \(\Pi\)
to a single \(\infty\)-test plan \(\ppi_2\). We proceed as follows: call
\(\Pi=(\ppi^k)_k\). Given any \(k\in\N\), there exists \(n_k\in\N\)
such that \(\ppi^k\) is concentrated on \(n_k\)-Lipschitz curves.
Then let us define
\[
\ppi^{k,i}\coloneqq\big({\rm restr}_{(i-1)/n_k}^{i/n_k}\big)_\#\ppi^k
\quad\text{ for every }i=1,\ldots,n_k.
\]
Therefore, we have that \(\ppi^{k,1},\ldots,\ppi^{k,n_k}\) are
\(\infty\)-test plans concentrated on \(1\)-Lipschitz curves.
Observe also that the family \(\Pi'\coloneqq\big\{\ppi^{k,i}\,:\,k\in\N,\,
i=1,\ldots,n_k\big\}\) satisfies \(W^{1,2}_{\Pi'}(\X)=W^{1,2}_\Pi(\X)\)
and \(|Df|_{\Pi',2}=|Df|_{\Pi,2}\) for every \(f\in W^{1,2}_{\Pi'}(\X)\).
Finally, let us define
\[
\eeta\coloneqq\sum_{k=1}^\infty\sum_{i=1}^{n_k}\frac{\ppi^{k,i}}
{2^{k+i}\max\big\{{\rm Comp}(\ppi^{k,i}),1\big\}},\qquad
\ppi_2\coloneqq\frac{\eeta}{\eeta\big(C\big([0,1],\X\big)\big)}.
\]
By arguing as in \textsc{Step 3} of the proof of Theorem
\ref{thm:master_tp}, we can see that \(\ppi_2\) is a \(\infty\)-test
plan (concentrated on \(1\)-Lipschitz curves). Given that
\(W^{1,2}_{\sppi_2}(\X)=W^{1,2}_{\Pi'}(\X)\) and \(|Df|_{\sppi_2,2}=
|Df|_{\Pi',2}\) for every \(f\in W^{1,2}_{\sppi_2}(\X)\), the
statement finally follows from the identity \eqref{eq:master_tp_RCD_aux}.
\end{proof}
\begin{remark}{\rm
We point out that every \(2\)-test plan \(\ppi\) induced by the
regular Lagrangian flow associated with \(\nabla f\), as in
\eqref{eq:tp_associated_RLF}, \(2\)-represents the gradient of \(f\).
Indeed, for \((\ppi\otimes\mathcal L^1)\)-a.e.\ \((\gamma,t)\)
it holds that \(|\dot\gamma_t|=|\ppi'_t|(\gamma)=
|\e_t^*\nabla f|(\gamma)=|Df|_2(\gamma_t)\) and
\(\frac{\d}{\d t}f(\gamma_t)=|Df|_2^2(\gamma_t)\), whence
\begin{align*}
\frac{{\rm E}_{2,t}(\gamma)}{t}&=\bigg(\fint_0^t|\dot\gamma_s|^2\,\d s
\bigg)^{\frac{1}{2}}=\bigg(\fint_0^t|Df|_2^2\circ\e_s\,\d s\bigg)
^{\frac{1}{2}}(\gamma),\\
\bigg(\frac{f\circ\e_t-f\circ\-e_0}{{\rm E}_{2,t}}\bigg)(\gamma)&=
\frac{t}{{\rm E}_{2,t}(\gamma)}\fint_0^t\frac{\d}{\d s}f(\gamma_s)\,\d s
=\frac{t}{{\rm E}_{2,t}(\gamma)}\fint_0^t|Df|_2^2(\gamma_s)\,\d s\\
&=\bigg(\fint_0^t|Df|_2^2\circ\e_s\,\d s\bigg)^{\frac{1}{2}}(\gamma)
\end{align*}
for every \(t\in(0,1)\) and \(\ppi\)-a.e.\ \(\gamma\). By recalling
Proposition \ref{prop:cont_f_et}, we conclude that the plan \(\ppi\)
\(2\)-represents the gradient of \(f\), as claimed above. This means
that Theorem \ref{thm:master_tp_RCD} could have been alternatively
proven by directly using the proof of Theorem \ref{thm:master_tp}.
\fr}\end{remark}
\begin{remark}\label{rmk:infty_tp_diff_p}{\rm
Suppose to have a metric measure space \((\X,\sfd,\mm)\) satisfying
the following property: given any \(p\in(1,\infty)\), there exists
a master \(q\)-test plan \(\ppi_q\) that is a \(\infty\)-test plan.
Then it can be readily checked that minimal \(p\)-weak upper
gradients are independent of \(p\). In light of Remark \ref{rmk:ineq_diff_p},
we deduce that there exist spaces where the above property fails.

On \({\sf RCD}(K,N)\) spaces the minimal weak upper gradients do not depend
on the exponent. If \(N\) is finite, then the space is (locally
uniformly) doubling and satisfies a (weak, local) Poincar\'{e} inequality,
thus the claim follows from the results of \cite{Cheeger00}; in the
infinite-dimensional case, it is proven in \cite{GH14}. According to this
observation, we might expect (or, at least, it is possible) that Theorem
\ref{thm:master_tp_RCD} can be generalised to all exponents \(p\in(1,\infty)\).
\fr}\end{remark}
\def\cprime{$'$} \def\cprime{$'$}

\end{document}